\documentclass[11pt]{amsart}

\usepackage{amsmath,amssymb,amsthm,tikz}
\usepackage[numbers,comma,square,sort&compress]{natbib}
\usepackage{mathabx}   %% made \widecheck work  
\usepackage{datetime2}
\usepackage{kbordermatrix}

\RequirePackage{color}
\RequirePackage[colorlinks,urlcolor=my-blue,linkcolor=my-red,citecolor=my-green]{hyperref}

\numberwithin{figure}{section}

\usepackage{enumerate}

\usepackage{scalerel}    %%%%  used first for scaling \bullet  
\usepackage{stmaryrd}   %%%% for integer interval notation 
\usepackage{mathabx}   %% made \widecheck work 

\definecolor{my-blue}{rgb}{0.0,0.0,0.6}
\definecolor{my-red}{rgb}{0.5,0.0,0.0}
\definecolor{my-green}{rgb}{0.0,0.5,0.0}
\definecolor{nicos-red}{rgb}{0.75,0.0,0.0}

\newtheorem{theorem}{\sc Theorem}[section]
\newtheorem{lemma}[theorem]{\sc Lemma}
\newtheorem{proposition}[theorem]{\sc Proposition}

\newtheorem{definition}[theorem]{\sc Definition}
\numberwithin{equation}{section}
\theoremstyle{remark}
\newtheorem{remark}[theorem]{Remark}

\newcommand{\be}{\begin{equation}}
\newcommand{\ee}{\end{equation}}
\newcommand{\beq}{\begin{equation}}
\newcommand{\eeq}{\end{equation}}

\providecommand{\abs}[1]{\vert#1\vert}

\def\cC{\mathcal{C}}

\def\cH{\mathcal{H}}

\def\cK{\mathcal{K}}

\def\cA{\mathcal{A}}
\def\cT{\mathcal{T}}  
\def\cY{\mathcal{Y}}

\def\kS{\mathfrak{S}}

\def\bE{\mathbb{E}} 

\def\bP{\mathbb{P}}

\def\bR{\mathbb{R}}
\def\bZ{\mathbb{Z}}

 \def\Z{\bZ}  \def\R{\bR}
  \def\evec{\mathbf{e}}  \def\uvec{\mathbf{u}}\def\vvec{\mathbf{v}}

\def\wvec{\mathbf{w}}

\def\w{\omega}

\def\e{\varepsilon}

\def\ddd{\displaystyle}

\def\mP{\mathbf{P}}
\def\m1{\mathbf{1}}

 \def\wt{\widetilde}    
 
\def\E{\bE}
\def\P{\bP} %% environment measure 
\def\OSP{(\Omega, \kS, \P)}
\def\OSPTh{(\Omega, \kS, \P, \Theta)}

\def\funct lp{L} %a function needed
\def\funct lpbar{\bar L} %a function needed

\def\Uset{\mathcal U}

\def\cA{\mathcal A}

\def\Gpp{G}

\def\gpp{g}

\def\h{{h}}

\DeclareMathOperator{\ri}{ri}    %rel interior  
\def\cH{\mathcal{H}}

\definecolor{darkgreen}{rgb}{0.0,0.5,0.0}
\definecolor{darkblue}{rgb}{0.0,0.0,0.3}
\definecolor{nicosred}{rgb}{0.65,0.1,0.1}
\definecolor{light-gray}{gray}{0.7}
\allowdisplaybreaks[1]

\usepackage{filemod}

\def\cif1{v}   %projection of CIF 
\def\deq{\overset{d}=}

\def\qedex{\hfill$\triangle$} 
   \def\tincr{t}   
\def\Xw{X}  %another letter for a weight 
\def\Yw{Y}  %another letter for a weight 
\def\bgeod#1#2{\mathbf b^{#1, #2}}

\def\bgeodu#1{\bgeod{\uvec}{#1}}

\def\swbgeod#1#2{\mathbf b^{{\rm sw},#1, #2}}

\def\swbgeodu#1{\swbgeod{\uvec}{#1}}

\def\dbgeod#1#2{\mathbf b^{*,#1, #2}}

\def\dbgeodu#1{\dbgeod{\uvec}{#1}}

\def\ageod#1#2{\pi^{#1, #2}}
\def\ageodu#1{\ageod{\uvec}{#1}}

\def\duale{\evec^*}  %%% shift between the lattice and its dual. 

\def\Bus#1{B^{#1}}

\newcommand\bbullet{{\raisebox{0.5pt}{\scaleobj{0.6}{\bullet}}}} %% looks good for paths x_\bbullet 
\newcommand\brbullet{{\raisebox{-1pt}{\scaleobj{0.5}{\bullet}}}} %%  for subscript paths x_\brbullet   need negative raise 
\newcommand\dbullet{{\scaleobj{0.7}{\bullet}}}   % same as \bbullet but without shift up

\def\wtcH{\wt\cH} 
 
 \RequirePackage{datetime} %comment out when arxiving or submitting

\allowdisplaybreaks[1]
\begin{document}
 
%\usdate
%\title[Corner growth model]{Coalescence of directed geodesics  for the exactly solvable corner growth model}
\title[Coalescence of geodesics]{Existence, uniqueness and coalescence of directed planar geodesics: proof via the  increment-stationary growth process}
%\footnotetext{\today,\ \currenttime}

\author[T.~Sepp\"al\"ainen]{Timo Sepp\"al\"ainen}
\address{Timo Sepp\"al\"ainen\\ University of Wisconsin-Madison\\  Mathematics Department\\ Van Vleck Hall\\ 480 Lincoln Dr.\\   Madison WI 53706-1388\\ USA.}
\email{seppalai@math.wisc.edu}
\urladdr{http://www.math.wisc.edu/~seppalai}
\thanks{%T.\ Sepp\"al\"ainen 
The author was partially supported by  National Science Foundation grants  DMS-1602486 and DMS-1854619 and by  the Wisconsin Alumni Research Foundation.} 

%\date{April 30, 2014}
%\date{September 22, 2014}
\date{\today} 
%\date{Last modified on \today\ at\ \filemodprinttime{\jobname}}
%\date{October 13, 2015}

\begin{abstract}
We present a proof of the almost sure existence, uniqueness and coalescence of directed semi-infinite geodesics in  planar growth models that is based on properties of an  increment-stationary version of the growth process. The argument is developed in the context of the exponential corner growth model. It uses coupling, planar monotonicity, and properties of the  stationary growth process to derive the existence   of Busemann functions, which in turn control geodesics.  This soft approach is  in some situations an alternative to the much-applied 20-year-old arguments of C.~Newman and co-authors. Along the way we derive some related results   such as the distributional equality of the directed geodesic tree and its dual, originally due to L.~Pimentel. 
\end{abstract}
\maketitle

%\tableofcontents

\section{Introduction}
%\section{Introduction and main results}
 
 \subsection{The corner growth model and its geodesics} 
The setting for the {\it planar corner growth model} (CGM) with {\it exponential weights} is  the following.    $\OSPTh$ is a measure-preserving $\Z^2$-dynamical system.  This means that  $\OSP$ is a probability space and  $\Theta=(\theta_x)_{x\in\Z^2}$ is a group of measurable bijections that acts  on $\Omega$   and  preserves $\P$: $\P(\theta_xA)=\P(A)$ for all events $A\in\kS$ and $x\in\Z^2$.    The  generic sample point  of $\Omega$ is denoted by $\w$.  The random weights   $\Yw=(\Yw_x)_{x\in\Z^2}$  are independent, identically distributed  (i.i.d.)  rate  1 exponentially distributed random variables   on $\Omega$ that  satisfy $\Yw_x(\w)=\Yw_0(\theta_x\w)$ for each $x\in\Z^2$ and almost every $\w\in\Omega$.  

%In general, a rate $\lambda$ exponentially distributed random variable $X$ satisfies   $\P\{X>t\}=e^{-\lambda t}$ for $t\ge 0$, abbreviated as $X\sim$ Exp$(\lambda)$.   

 The canonical choice is the product space $\Omega=\R_{\ge0}^{\Z^2}$ with translations $(\theta_x\w)_y=\w_{x+y}$, an i.i.d.\ product measure $\P$ and the coordinate process $\Yw_x(\w)=\w_x$.  

  The {\it last-passage percolation}  (LPP) process $\Gpp=\Gpp^Y$  is defined for $x\le y$ (coordinatewise order)  on $\Z^2$ by 
 	\be\label{v:GY7}  
	  \Gpp_{x,y}=\Gpp(x,y)=\max_{x_{\brbullet}\,\in\,\Pi_{x,y}}\sum_{k=0}^{\abs{y-x}_1}\Yw_{x_k}.
	\ee
	$\Pi_{x,y}$ is the set of  {\it up-right paths} $x_{\dbullet}=(x_k)_{k=0}^n$  that start at  $x_0=x$ and   end at $x_n=y$, with $n=\abs{y-x}_1$.    By definition, the increments of   an up-right path satisfy   $x_{k+1}-x_k\in\{\evec_1,\evec_2\}$.      A path can be equivalently characterized in terms of its vertices or its edges.  Both points of view are useful.  See Figure \ref{fig:cgm}  for an illustration.   The zero-length path case  is $G_{x,x}=\w_x$.    Our convention is that 
\be\label{v:Gfail} 	 \Gpp_{x,y}=-\infty\qquad   \text{if $x\le y$ fails. } \ee

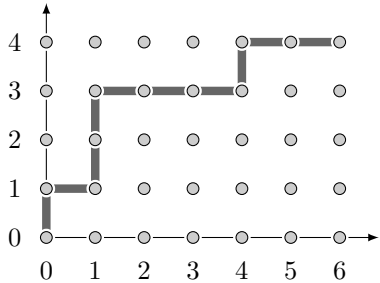
\begin{figure}[h] 
	\begin{center}
		\begin{tikzpicture}[>=latex, scale=0.65]  %[>=latex] 
%			\definecolor{sussexg}{rgb}{0,0.5,0.5}
%			\definecolor{sussexp}{rgb}{0.6,0,0.6} Uncomment for colorful version

			\definecolor{sussexg}{gray}{0.7}
			\definecolor{sussexp}{gray}{0.4} %comment out for colorful version

			%AXES
			\draw[<->](0,4.8)--(0,0)--(6.8,0);
			
			% PATH
			\draw[line width = 3.2 pt, color=sussexp](0,0)--(0,1)--(1,1)--(1,3)--(3,3)--(4,3)--(4,4)--(6,4);
			
			% LATTICE
			\foreach \x in { 0,...,6}{
				\foreach \y in {0,...,4}{
					\fill[color=white] (\x,\y)circle(1.8mm); 
					% Above: Extra layer of circles to make the path disconnected at sites. Comment out for a different effect
					
					\draw[ fill=sussexg!65](\x,\y)circle(1.2mm);
				}
			}	
			
			%LABELS
			\foreach \x in {0,...,6}{\draw(\x, -0.3)node[below]{\small{$\x$}};}
			\foreach \x in {0,...,4}{\draw(-0.3, \x)node[left]{\small{$\x$}};}	
			
		\end{tikzpicture}
	\end{center}
	\caption{ \small  An example of an up-right path from $(0,0)$ to $(6,4)$ on the lattice $\Z^2$. }\label{fig:cgm} 
  \end{figure}

The   shape function of the exponential CGM has been known since the seminal paper of Rost \cite{rost}:  
 \be\label{g}
 \gpp(\xi)= \bigl(\sqrt \xi_1+\sqrt \xi_2\,\bigr)^2\qquad \text{for  }  \xi=(\xi_1, \xi_2)\in\R_{\ge0}^2 .  
 \ee
The shape theorem is the law of large numbers of the LPP process, uniform in all directions   (Theorem 5.1 in \cite{mart-04}, Theorem 3.5 in \cite{sepp-cgm-18}):
 
\begin{theorem}\label{t:sh} 
Given $\e>0$, there exists a $\P$-almost surely finite  random variable  $K$ such that 
\be\label{sh89} \abs{ \Gpp_{0,x}-\gpp(x)}  \le \e\abs{x}_1 
\quad \text{ for all $x\in\Z_{\ge 0}^2$ such that $\abs{x}_1\ge K$.}  
\ee
\end{theorem}

An up-right path $(x_i)_{i\in I}$ indexed by a finite or infinite subinterval $I\subset\Z$  is a  {\it geodesic}  if it is the maximizing path between any two of its points: 
\be\label{geod19}  
	  \Gpp_{x_k, x_\ell}= \sum_{i=k}^{\ell}\Yw_{x_i} \qquad \text{for all $k<\ell$ in } I.
	\ee
Since the weight distribution is continuous, maximizing paths between any two points are unique $\P$-almost surely.  
A geodesic $(x_i)_{i\in\Z_{\ge0}}$ indexed by nonnegative integers  is called a  {\it semi-infinite geodesic} started at $x_0$, and a geodesic $(x_i)_{i\in\Z}$ indexed by the entire  integer line   is a  {\it bi-infinite geodesic}.     A semi-infinite or bi-infinite geodesic $x_\dbullet$ % $(x_i)_{i\in\Z_{\ge0}}$ 
is {\it $\uvec$-directed}  if   %$\abs{x_n}_1\to\infty$ and 
$x_n/n\to\uvec$ as $n\to\infty$.  

\subsection{The purpose of the paper and its relation to past work} 
We address  the   existence, uniqueness and coalescence of semi-infinite  geodesics in a given direction $\uvec$.  The results  themselves  are not new.  The purpose   is to present an alternative proof of these known results.    

% We explain the motivation  of this work and its  relationship to the existing literature. 

Already for about  two decades,  geodesics and the closely related Busemann functions have been important in the study of first- and last-passage growth models,  and recently also in positive-temperature polymer models.   Proof techniques for the  existence, uniqueness and coalescence of semi-infinite directed geodesics  developed by  C.~Newman and co-authors \cite{howa-newm-01, lice-newm-96, newm-icm-95}  have played a  central role in this work.  
This approach controls  the   wandering of geodesics with estimates that rely on   assumptions on the limit shape,  to show that each direction has a geodesic and each geodesic has a direction.  
% so show that geodesics are unique.  Uniqueness implies that if geodesics intersect, they coalesce. From this they are shown to coalesce almost surely,  through 
Almost sure coalescence is shown by a modification argument followed by a Burton-Keane type lack of space argument.   

 These techniques have been applied to great benefit in many models where sufficient solvability  or symmetries enable the verification of the  hypotheses imposed on the limit shape.  In the exponential CGM this proof  was implemented  by P.~A.~Ferrari and L.~Pimentel \cite{ferr-pime-05}.    Examples of applications   to LPP and positive-temperature polymers with quadratic limit shapes appear in \cite{bakh-cato-khan-14, bakh-li, cato-pime-11}.  
 
 %  In a slightly  different direction,  the Licea-Newman coalescence argument was adapted to the so-called cocycle    geodesics in the general CGM in the arXiv preprint  \cite{geor-rass-sepp-lppgeo}. 

%What was the shortcoming of Licea-Newman?  Where would this argument be better? When stationarity is missing, for example in models with inhomogeneous parameters. Or when limit shape does not possess curvature.    

The proof developed in this paper  replaces the  estimates that control geodesics and the technical modification arguments  with a softer proof that comes  from structural properties.     
%The alternative arguments presented here probably cannot cover all the applications of Newman's arguments because sometimes the geometric information about the curvature of the limit shape is the main piece of information available. 
This  proof   can  be substituted for  Newman's proof in cases where sufficiently tractable increment-stationary versions of the growth process can be constructed.    This may  be possible in some  situations where shift-invariance and curvature  are  not available.  This  would be the case for example in models with inhomogeneous parameters, such as those whose limit shapes are studied in  \cite{emra-16-ecp}.  

As a consequence of our development  we establish 
Pimentel's distributional equality  \cite{pime-16}  of the directed geodesic tree and its dual,  without recourse to mappings between the CGM and the  totally asymmetric simple exclusion process (TASEP).   It is useful to develop a proof of this result  within the context of the growth model itself,  for the purpose of  extension  to growth models and  polymer models that are not connected to particle systems.  Pimentel   \cite{pime-16} used this duality to derive bounds on coalescence times. 

\subsection{Other related work} 
Recent work where coalescence of geodesics figures prominently include \cite{geor-rass-sepp-17-geod} on the CGM with general weights and \cite{damr-hans-14, damr-hans-17} on undirected first-passage percolation.  These papers prove coalescence with  the Licea-Newman argument.  The proof given here does  not presently apply to the models  studied there    because the properties of their  Busemann functions are not yet sufficiently well understood.  

Chaika and Krishnan \cite{chai-kris-arxiv} consider paths on a lattice defined by an ergodic field of nearest-neighbor ``arrows'', or local gradients.  They use ergodicity and a very general volume argument to show that if coalescence fails, bi-infinite paths exist.  Theirs would be an alternative  proof of  the (iii)$\Longrightarrow$(i) implication for Busemann geodesics  in Lemma \ref{bgeod-lm8}  below.   Our argument is more model-specific and uses the equal distribution of Busemann geodesics and their duals. 

 \subsection{Notation and conventions}  
Points $x=(x_1,x_2),y=(y_1,y_2)\in\R^2$ are ordered coordinatewise: $x\le y$ iff  $x_1\le y_1$  and $x_2\le y_2$.    The $\ell^1$ norm is $\abs{x}_1=\abs{x_1}+\abs{x_2}$.   A path as a sequence of points $(x_k)_{k=0}^n$ can be denoted by $x_\dbullet$ or by $x_{0,n}$.  Subscripts indicate restricted subsets of the reals and integers: for example 
% $\R_{>1}=(1,\infty)$ and  
$\Z_{>0}=\{1,2,3,\dotsc\}$ and $\Z_{>0}^2=(\Z_{>0})^2$ is the positive first quadrant of the planar integer lattice.    Boldface  notation for special vectors: $\evec_1=(1,0)$, $\evec_2=(0,1)$,   and members of the simplex %$[\evec_2,\evec_1]
$\Uset=\{t\evec_1+(1-t)\evec_2: 0\le t\le 1\}$ are denoted by $\uvec$, $\vvec$ and $\wvec$.  
%For $n\in\Z_{>0}$, $[n]=\{1,2,\dotsc,n\}$, with the convention that $[n]=\varnothing$ for $n\in\Z_{\le0}$.  A finite  integer interval is denoted by $\lzb m,n\rzb=\{m,m+1,\dotsc,n\}$,  and  $\lzb m,\infty\lzb=\{m,m+1, m+2, \dotsc\}$.  
   For $0<\alpha<\infty$, $X\sim$ Exp$(\alpha)$ means that random variable $X$ has exponential distribution with rate $\alpha$, in other words $P(X>t)=e^{-\alpha t}$ for $t>0$ and $E(X)=\alpha^{-1}$.  Functional arguments can be equivalently written as subscripts, as in $B(x,y,\w)=B_{x,y}(\w)$. 
   
   \subsection{Acknowledgements}  This paper benefited from numerous discussions and collaborations   over the years,  especially with  E.~Emrah, N.~Georgiou, C.~Janjigian, A.~Krishnan, F.~Rassoul-Agha, and A.~Y{\i}lmaz.  The exposition was improved by four anonymous referees. 

\section{Main results on directed semi-infinite geodesics} \label{s:main}
 Here is a restatement of the assumption: 
\be\label{ass-B} 
\begin{aligned}  
&\text{$\OSPTh$ is a measure-preserving $\Z^2$-dynamical system and   $\Yw=(\Yw_x)_{x\in\Z^2}$   }\\ 
&\text{are i.i.d.\ Exp(1) random variables on  $\Omega$ that    satisfy $\Yw_x(\w)=\Yw_0(\theta_x\w)$ $\P$-a.s.}  
  \end{aligned}\ee

  The set of possible asymptotic velocities or direction vectors for semi-infinite up-right paths is $\Uset=\{(t, 1-t): 0\le t\le 1\}$, 
  %$\Uset=[\evec_2,\evec_1]=\{t\evec_1+(1-t)\evec_2: 0\le t\le 1\}$, 
  with relative interior $\ri\Uset=\{(t, 1-t): 0< t<1\}$.   
%  $\ri\Uset=\{t\evec_1+(1-t)\evec_2: 0< t<1\}$.   
 
%Let $\pi^{x,y}_\dbullet \in\Pi_{x,y}$ denote the (almost surely unique)  maximizing path for $\Gpp_{x,y}$ defined by \eqref{v:G}. That is, the last-passage value satisfies 
%\be\label{v:Gpi}  
%	  \Gpp_{x,y}= \sum_{i=0}^{\abs{y-x}_1}\w_{\pi^{x,y}_i}.
%	\ee
%Even though last-passage percolation seeks to maximize rather than minimize path weight, $\pi^{x,y}_\bbullet$ is often called the {\it geodesic} from $x$ to $y$. 

We start with the results that are almost surely valid for all geodesics and directions.

\begin{theorem}\label{t:main-1}  Assume \eqref{ass-B}.   Then the following statements hold with $\P$-probability one. 

{\rm (i)}  Each semi-infinite geodesic is $\uvec$-directed for some $\uvec\in\Uset$.   

{\rm (ii)} For $r\in\{1,2\}$  and each $x\in\Z^2$,  $\{x_k=x+k\evec_r\}_{k\in\Z_{\ge0}}$ is the only semi-infinite geodesic that satisfies $x_0=x$ and $\varliminf_{k\to\infty} k^{-1} x_k\cdot\evec_{3-r}=0$. 
% Similarly,  $\{y_k=x+k\evec_2\}_{k\in\Z_{\ge0}}$ is the only semi-infinite geodesic that satisfies $y_0=x$ and $\varliminf_{k\to\infty} k^{-1} x_k\cdot\evec_1=0$.

% For $r\in\{1,2\}$ and $x\in\Z^2$, the only $\evec_r$-directed semi-infinite geodesic is $x_k=x+k\evec_r$.  

{\rm (iii)}  For each $\uvec\in\Uset$  and $x\in\Z^2$   there exists a  $\uvec$-directed semi-infinite geodesic that starts at $x$.  
\end{theorem}

Parts  (i)--(ii) together say that except for the trivial geodesics $x_k=x+k\evec_r$ with constant increments, every semi-infinite geodesic is directed towards a vector $\uvec$ in the {\it interior}  of the first quadrant.  

The next theorem states properties that hold almost surely for a given direction $\uvec$. 

\begin{theorem}\label{t:main1}  Assume \eqref{ass-B}.   
Fix $\uvec\in\ri\Uset$.  Then the following statements hold  with $\P$-probability one. 

{\rm (i)}  For each $x\in\Z^2$ there exists a unique $\uvec$-directed semi-infinite geodesic $\ageodu{x}=(\ageodu{x}_k)_{k\in\Z_{\ge0}}$ with initial point  $\ageodu{x}_0=x$. Each point $\ageodu{x}_k$ is a Borel function of the weights $\Yw$.   For each pair $x,y\in\Z^2$ these geodesics coalesce: that is,  there exists $z\in\Z^2$ such that %{\rm(}as sets{\rm)} 
 $\ageodu{x}\cap \ageodu{y}=\ageodu{z}$. 

{\rm (ii)} There is no bi-infinite geodesic in direction $\uvec$. 
\end{theorem}

Let $\cT_\uvec$ be the tree of all the $\uvec$-directed semi-infinite geodesics $\{\ageodu{x}: x\in\Z^2\}$.  That is, 
\be\label{T1}   \cT_\uvec=\bigcup_{x\in\Z^2} \ageodu{x}  \ee
 when we regard a geodesic as a collection of edges.  

 The {\it dual lattice}  $\Z^{2*}$ of $\Z^2$ 
 % consists of the vertices $\{(k+\tfrac12, \ell+\tfrac12): k,\ell\in\Z\}$  and the nearest-neighbor edges
    is obtained  by translating all the vertices and (nearest-neighbor) edges of  $\Z^2$ by the vector   $\duale=\tfrac12(\evec_1+\evec_2)=(\tfrac12,\tfrac12)$.   %   $\Z^{2*}=\Z+\duale$.   
       An edge  of $\Z^2$ and an edge   of $\Z^{2*}$ are  {\it dual}  if  they  cross each other or, equivalently,  intersect at their midpoints.   The unique dual of an edge $e$ of $\Z^2$ is denoted by $e^*$, and  similarly $f^*$ denotes the dual of an edge $f$ of $\Z^{2*}$.   In particular,    if   $e=\{x-\evec_k,x\}$ then $e^*=\{x-\duale, x-\duale+\evec_{3-k}\}$, and $e^{**}=e$.   
       
        The dual graph $\cT_\uvec^*$ of the tree $\cT_\uvec$ is defined through the edge duality:  
     \be\label{T*}  
     e^*\in\cT_\uvec^* \quad\text{if and only if} \quad e\notin\cT_\uvec. 
     %   edge   $e^*$  lies in $\cT_\uvec^*$ if and only if   $e$  does {\it not} lie in $\cT_\uvec$.  
      \ee  
      Move the dual graph $\cT_\uvec^*$ back on the original lattice by defining  the graph 
      \be\label{Ttil} \wt\cT_\uvec=-\duale-\cT_\uvec^*.  \ee
      That is,  %for $x\in\Z^2$ and $k\in\{1,2\}$, 
      edge  $\{x-\evec_k, x\}\in\wt\cT_\uvec$   %(-\duale-\cT_\uvec^*)$ 
 if and only if edge $\{-x-\duale, -x-\duale+\evec_k\}\in\cT_\uvec^*$.    The point of the next theorem is that $\wt\cT_\uvec$ is also a tree of directed geodesics of an exponential CGM. 

\begin{theorem}\label{t:main2}  Assume \eqref{ass-B}.   Fix $\uvec\in\Uset$.  Then there exists a collection   $\wt Y^\uvec=(\wt Y^\uvec_x)_{x\in\Z^{2}}$  of i.i.d.\ {\rm Exp(1)}  weights on $\OSP$  with these properties.  

{\rm(i)}    $\wt Y^\uvec$ is a Borel function of the weights $Y$ in \eqref{ass-B}  and $\wt Y^\uvec_x(\theta_y\w)=\wt Y^\uvec_{x-y}(\w)$ $\forall x,y\in\Z^2$. 

{\rm(ii)}  $\P$-almost surely   $\wt\cT_\uvec$   is the tree of the  unique $\uvec$-directed semi-infinite geodesics of the LPP process  $\Gpp^{\wt Y^\uvec}$ defined as in  \eqref{v:GY7} with $Y$ replaced by $\wt Y^\uvec$.  

In particular, the tree   $\wt\cT_\uvec$ is    equal in  distribution to $\cT_\uvec$.  The dual graph $\cT^*_\uvec$  is also  $\P$-almost surely a tree. 
\end{theorem} 

 The equality in distribution of $\cT_\uvec$ and the (shifted and reflected) dual graph $\cT^*_\uvec$  was  originally proved by Pimentel (Lemma 2 in \cite{pime-16}).   The weights $\wt Y^\uvec$ are defined in \eqref{Ytil4} below. 

\medskip 

As the final main results, we record some immediate  consequences of the properties of  Busemann functions, to be described in the next section. Distributional properties of the geodesic tree $\cT_\uvec$ depend on a  real parameter $\alpha\in(0,1)$  that is in   bijective correspondence 
    with the  direction  $\uvec=(u_1,1-u_1)\in\ri\Uset$. This bijection is defined    by the equations 
 \be\label{u-a}    \uvec=\uvec(\alpha)=\biggl(    \frac{\alpha^2}{(1-\alpha)^2+\alpha^2}\,,  \frac{(1-\alpha)^2}{(1-\alpha)^2+\alpha^2}\biggr)
 \ \Longleftrightarrow \ 
 \alpha= \alpha(\uvec)= \frac{\sqrt{u_1}}{\sqrt{u_1}+ \sqrt{1-u_1}}. 
  \ee
For example, $\alpha$ gives  the distribution of the first step of the geodesic:
\be\label{a1}
\P\{ \ageodu{x}_1=x+\evec_1\}=\alpha \qquad\forall x\in\Z^2. 
\ee
This statement is proved after Lemma \ref{Bgeod-lm}, after the proof of Theorem \ref{t:main-1}.   Note however that the density of $\evec_1$ steps along the $\uvec$-directed semi-infinite geodesic is $u_1$,  which is  different from $\alpha$,  except in the special case $u_1=\alpha=\tfrac12$.  This points to the fact that understanding distributional properties along a geodesic is challenging.  It is much easier to capture properties transversal to geodesics, as the next theorem illustrates.

Call a point  $z\in\Z^2$ a {\it source} if $z$ does not lie on $\ageodu{x}$ for any $x\ne z$. Call $z$ a {\it coalescence point} if  there exist   $x\ne y$ in $\Z^2\setminus\{z\}$ such that  $\ageodu{z}=  \ageodu{x}\cap\ageodu{y}$.    Equivalently,   $z$ is a source if $\ageodu{z-\evec_1}_1=z-\evec_1+\evec_2$ and $\ageodu{z-\evec_2}_1=z-\evec_2+\evec_1$, while $z$ is a coalescence point if  $\ageodu{z-\evec_1}_1=\ageodu{z-\evec_2}_1=z$.   To complete the list of possibilities,  call $z$ a {\it horizontal point} if $\ageodu{z-\evec_1}_1=z$ but $\ageodu{z-\evec_2}_1=z-\evec_2+\evec_1$, and a {\it vertical  point} if $\ageodu{z-\evec_1}_1=z-\evec_1+\evec_2$ but $\ageodu{z-\evec_2}_1=z$.   See Figure \ref{fig:schv} for an illustration. 

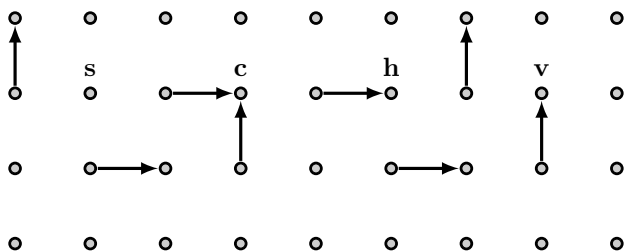
\begin{figure}%[h] 
\vspace{0.2cm} 
	\begin{center}
		\begin{tikzpicture}[>=latex, scale=0.5]  %[>=latex] 
%			\definecolor{sussexg}{rgb}{0,0.5,0.5}
%			\definecolor{sussexp}{rgb}{0.6,0,0.6} Uncomment for colorful version

			\definecolor{sussexg}{gray}{0.7}
			%\definecolor{sussexp}{gray}{0.4} %comment out for colorful version

			%AXES
			%\draw[<->](0,4.8)--(0,0)--(6.8,0);
			
			% PATH
			%\draw[line width = 3.2 pt, color=sussexp](0,0)--(0,1)--(1,1)--(1,3)--(3,3)--(4,3)--(4,4)--(6,4);
			
			% LATTICE
			\foreach \x in { 0,2,4,6,8,10,12,14,16}{
				\foreach \y in {0,2,4,6}{
					\fill (\x,\y)circle(1.8mm); 
					% Above: Extra layer of circles to make the path disconnected at sites. Comment out for a different effect
					
					\draw[ fill=sussexg!65](\x,\y)circle(1.2mm);
				}
			}	
			
			\draw[->][line width=1.2 pt] (4.2,4)-- (5.8,4); 
			\draw[->][line width=1.2 pt] (6,2.2)-- (6,3.8); 
			
			\draw[->][line width=1.2 pt] (2.2,2)-- (3.8,2); 
			\draw[->][line width=1.2 pt] (0,4.2)-- (0,5.8); 
			
			\draw[->][line width=1.2 pt] (8.2,4)-- (9.8,4); 
			\draw[->][line width=1.2 pt] (10.2,2)-- (11.8,2); 
			
			\draw[->][line width=1.2 pt] (12,4.2)-- (12,5.8); 
			\draw[->][line width=1.2 pt] (14,2.2)-- (14,3.8); 
			
			%LABELS
			\draw(2,4.2)node[above]{\small{$\mathbf s$}};
			\draw(6,4.2)node[above]{\small{$\mathbf c$}};
			\draw(10,4.2)node[above]{\small{$\mathbf h$}};
			\draw(14,4.2)node[above]{\small{$\mathbf v$}};
			%\draw{\small{$a$}}(2,4.3);
			
			%\foreach \x in {2,6,10,14}{\draw(\x, 4)node[below]{\small{$\x$}};}
			%\foreach \x in {0,...,4}{\draw(-0.3, \x)node[left]{\small{$\x$}};}	
			
		\end{tikzpicture}
	\end{center}
	\caption{ \small  An example of a source ($\mathbf s$),  a coalescence point ($\mathbf c$), a horizontal point ($\mathbf h$), and a vertical point ($\mathbf v$). The arrows point from $x$ to $\ageodu{x}_1$. There is an arrow from each vertex $x$ but only the arrows needed for the definitions are displayed in the figure. }\label{fig:schv} 
 \vspace{0.1cm}  \end{figure}

Fix an antidiagonal  $\cA=\{ (N+j,-j): j\in\Z\}$ of the lattice $\Z^2$, for some $N\in\Z$.   Let 
$\xi_j$ be the random variable that takes one of the values $\{s,c,h,v\}$  to record whether point $(N+j,-j)$ is a source, a coalescence point, a horizontal point, or a vertical point.

%  Duality transforms sources to coalescence points and vice versa.  Each point $x$ is a coalescence point with probability $\alpha(1-\alpha)$ and a source with probability $\alpha(1-\alpha)$. 

\begin{theorem}\label{t:main3}  Assume \eqref{ass-B}.   Fix $\uvec\in\Uset$ and let $\alpha=\alpha(\uvec)$.  
Then $\{ \xi_j\}_{j\in\Z}$ is a stationary  Markov chain  with state space $\{s,c,h,v\}$,  transition matrix 
\be\label{mc}
\mP=\kbordermatrix{
& s&c&h&v \\  s & 0 & 1-\alpha & \alpha &0 \\ c &\alpha &0 &0 &1-\alpha\\ h & 0 & 1-\alpha & \alpha & 0  \\ v &\alpha &0 &0 &1-\alpha 
}   
\ee
and invariant distribution 
\[  \mu(s)=\mu(c)=\alpha(1-\alpha), \quad \mu(h)=\alpha^2, \quad \mu(v)=(1-\alpha)^2.  \] 
\end{theorem}

In particular,  both sources and coalescence points of  semi-infinite geodesics   in direction $\uvec=(u_1, 1-u_1)$   have density 
\[    \alpha(\uvec)(1-\alpha(\uvec))=  \frac{\sqrt{u_1(1-u_1)}}{\bigl(\sqrt{u_1}+ \sqrt{1-u_1}\,\bigr)^2}  \]
on the lattice.  
This density is maximized at $1/4$ by the diagonal direction $\uvec=(\tfrac12,\tfrac12)$.

%\subsection*{Organization of the proofs} 

\subsection*{Organization of the rest of the paper}    As mentioned, the purpose of the paper is to present  a particular proof of Theorems \ref{t:main-1}--\ref{t:main2}.   This proof has three main steps.  
%The following are the  three main steps of the proofs of Theorems \ref{t:main1} and \ref{t:main2}. 
 \begin{enumerate}[(i)] \itemsep=2pt
\item
 Construction of the increment-stationary LPP process.
\item Proof of the existence and properties of Busemann functions, by using couplings with the increment-stationary LPP  and monotonicity.
\item Control of geodesics with the Busemann functions. 
\end{enumerate}
%(i) Construction of the increment-stationary LPP process.
%(ii) Proof of the existence and properties of Busemann functions, by using couplings with the increment-stationary LPP  and monotonicity.
%(iii) Control of geodesics with the Busemann functions. 
Full details of steps (i) and (ii) are omitted from this paper because these steps are spelled out in  lecture notes \cite{sepp-cgm-18}.  We review these arguments briefly in Section \ref{s:bus1}.  The work of this paper goes towards step (iii).    This is done in  Section \ref{s:bg} that develops Busemann geodesics and proves the  theorems of  Section \ref{s:main}.   A final Section \ref{s:cif} relates the geodesics constructed in Section \ref{s:main} to competition interfaces.  
%\ref{t:main1} and \ref{t:main2}. 
% Proofs of the results rely on a special family of geodesics called Busemann geodesics.  We show that for a  given $\uvec$,   the semi-infinite geodesics are $\P$-almost surely Busemann geodesics.  Then we prove the results for the Busemann geodesics.   
 % The narrative  is   self-contained once the  preliminaries reviewed in Section \ref{s:prel} are in place: the shape theorems for the LPP process and integrable covariant cocycles,  and the existence and properties of Busemann functions.   

% \subsection{Existence and properties of Busemann functions} 
  \section{Increment-stationary LPP and Busemann functions}  
  \label{s:bus1}
  
  \subsection{Preliminaries} 
A {\it down-right path} is a bi-infinite sequence $\cY=(y_k)_{k\in\Z}$  in   $\Z^2$ such that    $y_k-y_{k-1}\in\{\evec_1,-\evec_2\}$ for all $k\in\Z$.  The  lattice decomposes  into a disjoint union  $\Z^2=\cH^-\cup\cY\cup\cH^+$  where   the two regions are 
 \be\label{H-}   \cH^-=\{x\in \Z^2:  \exists j\in\Z_{>0}  \text{ such that  } x+j(\evec_1+\evec_2)\in \cY\} \ee
 to the left of  and below $\cY$   and 
 \be\label{H+} \cH^+=\{x\in \Z^2:  \exists j\in\Z_{>0}  \text{ such that } x-j(\evec_1+\evec_2)\in \cY\}  \ee
 to the right of and above $\cY$. 

It will be convenient to summarize certain  properties  of systems of exponential weights  in the following definition.
 
 \begin{definition} \label{v:d-exp-a} 
Let $0<\alpha<1$.   A  stochastic process 
$ \{ \zeta_{x} ,\,   I_{x}, \, J_{x},  \,\eta_{x}  :  x\in \Z^2\}$ 
%  \be\label{IJw800} 
%  \{ \zeta_{x} ,\,   I_{x}, \, J_{x},  \,\eta_{x}  :  x\in \Z^2\} 
% %  \{ \w_{x+\evec_1+\evec_2} ,\,   I_{x+\evec_1}, \, J_{x+\evec_2},  \,\wc\w_{x}  :  x\in u+\Z_{\ge0}^2\} . 
%   \ee
 is an {\rm exponential-$\alpha$  last-passage percolation system}  if the following properties  {\rm (a)--(b)} hold. \\[-7pt] 
 \begin{enumerate}[{\rm(a)}]\itemsep=7pt 
 \item  The process is stationary under lattice translations  and has  marginal distributions  
 \be\label{w19}\begin{aligned}
\zeta_{x},\, \eta_x \sim \text{\rm Exp}(1),  \quad I_{x}\sim \text{\rm Exp}(\alpha)  , \quad \text{ and }  \quad  J_{x}\sim \text{\rm Exp}(1-\alpha)  .  %j\in\N \\  
\end{aligned}\ee 
  For any down-right path $\cY=(y_k)_{k\in\Z}$  in   $\Z^2$, the random variables 
 \be\label{v:Y4}   
 \{\eta_z:  z\in\cH^-\}, \quad  \{  \tincr(\{y_{k-1},y_k\}): k\in\Z\}, \quad\text{and}\quad 
 \{  \zeta_x: x\in\cH^+\} 
%  \{ \, \wc\w_z,  \, \tincr(\{y_{k-1},y_k\}),  \, \w_x:  z\in\cH^-, \, k\in\Z, \, x\in\cH^+\} 
 \ee
 are all mutually independent, where the undirected edge variables $\tincr(e)$ are defined  as 
  \be\label{v:tincr6}   \tincr(e)= \begin{cases}  I_x &\text{if $e=\{x-\evec_1,x\}$}\\ J_x &\text{if $e=\{x-\evec_2,x\}$.} \end{cases}  \ee

 \item  The following equations  are in force at all $x\in\Z^2$: 
 \begin{align}
\label{IJw5.1.7}  \eta_{x-\evec_1-\evec_2}&=I_{x-\evec_2}\wedge J_{x-\evec_1} % \qquad \text{for $x\in u+\Z_{\ge 0}^2$}
\\
\label{IJw5.2.7}  I_x&=\zeta_x+(I_{x-\evec_2}-J_{x-\evec_1})^+ % \qquad \text{for $x\in u+\Z_{> 0}^2$}
\\
\label{IJw5.3.7} J_x&=\zeta_x+(I_{x-\evec_2}-J_{x-\evec_1})^-. % \qquad \text{for $x\in u+\Z_{> 0}^2$}  
\end{align}

 \end{enumerate}
\qedex\end{definition} 

Equations \eqref{IJw5.2.7}--\eqref{IJw5.3.7}  imply this counterpart of \eqref{IJw5.1.7}:  
\be \label{IJw5.4.7}  \zeta_{x}=I_{x}\wedge J_{x}.  \ee

 An exponential-$\alpha$ LPP  system can be constructed explicitly in a quadrant as follows.    Assume given independent weights $\{I_{i\evec_1}:i\ge 1\}$ on the $x$-axis,   $\{J_{j\evec_2}:j\ge 1\}$ on the $y$-axis,  and $\{\zeta_x: x\in\Z^2_{>0}\}$ in the bulk (interior) of the first quadrant, all with marginal distributions \eqref{w19}.   Use equations \eqref{IJw5.1.7}--\eqref{IJw5.3.7} to define inductively in the northeast direction weights  $\{\eta_{x-\evec_1-\evec_2}, I_x, J_x: x\in\Z^2_{>0}\}$.   Then property (a)  from Definition \ref{v:d-exp-a} above can be verified inductively.   Now   $\{ \zeta_{x+\evec_1+\evec_2} ,\,   I_{x+\evec_1}, \, J_{x+\evec_2},  \,\eta_{x}  :  x\in \Z_{\ge0}^2\}$ is an   exponential-$\alpha$ LPP  system restricted to a  quadrant.

 Furthermore, if we define the LPP process $\{G^\alpha_x:  x\in\Z^2_{\ge0}\}$ by   $G^\alpha_0=0$, 
 \be\label{G-a1}   G^\alpha_{k\evec_1}=\sum_{i=1}^k I_{i\evec_1}\ \text{ for }k\ge 1, 
 \quad
 G^\alpha_{\ell\evec_2}=\sum_{j=1}^\ell J_{j\evec_2}\ \text{ for }\ell\ge 1,   \ee
 and inductively 
 \be\label{G-a}     G^\alpha_x = \zeta_x +  G^\alpha_{x-\evec_1} \vee G^\alpha_{x-\evec_2} \quad\text{for } x\in\Z^2_{>0}, \ee 
then $I$ and $J$ are the increments:
\be\label{G-a2}   I_x= G^\alpha_x -  G^\alpha_{x-\evec_1}   \quad\text{and}\quad 
J_x= G^\alpha_x -  G^\alpha_{x-\evec_2}.  \ee
 All this is elementary to verify and contained in Theorem 3.1 of \cite{sepp-cgm-18}.     $\{G^\alpha_x:  x\in\Z^2_{\ge0}\}$ is an {\it increment-stationary} LPP process.  
 
\medskip

To produce an  exponential-$\alpha$  LPP system   on the full lattice as a function of the i.i.d.\ weights $Y$ of assumption \eqref{ass-B},  we take limits of LPP increments in the direction $\uvec(\alpha)$ determined by \eqref{u-a}.   For the statement we need a couple  more definitions. 

 Define an order among  direction vectors  $\uvec=(u_1, 1-u_1)$ and  $\vvec=(v_1,1-v_1)$ in $\Uset$ according to the $\evec_1$-coordinate:   
  \be\label{u-order} \uvec\prec\vvec\quad\text{if}\quad u_1< v_1.  \ee
%  with weak ordering denoted by $\uvec\preceq\vvec$.     
  Geometrically:   $\uvec\prec\vvec$ if $\vvec$ is below  and to the right of  $\uvec$.  
    Bijection \eqref{u-a} preserves this order.

\begin{definition}%[Cocycles]
\label{def:cK}
A measurable function $B:\Omega\times\Z^2\times\Z^2\to\R$ is   a {\rm covariant   cocycle}
if it satisfies these two   conditions for $\P$-a.e.\ $\w$ and  all $x,y,z\in\Z^2$:
\begin{align*}
B(\w, x+z,y+z)&=B(\theta_z\w, x,y)  \qquad \text{{\rm(}stationarity{\rm)}}  \\
B(\w,x,y)+B(\w,y,z)&=B(\w, x,z) \qquad  \;\  \ \text{{\rm(}additivity{\rm)}.}
\end{align*}
$\cK$ denotes the space of covariant cocycles $B$ such that  $\E\abs{B(x,y)} <\infty$   $\forall x,y\in\Z^2$. 
\end{definition}

\subsection{Busemann functions} 
Existence and properties of Busemann functions are summarized in the next theorem.

 \begin{theorem}\label{t:buse}  Assume \eqref{ass-B}.  Then  
 for each  $\uvec\in\ri\Uset$  there exist a covariant   cocycle $\Bus{\uvec}=(\Bus{\uvec}_{x,y})_{x,y\in\Z^2}$  and a family of random weights  $\Xw^\uvec=(\Xw^{\uvec}_x)_{x\in\Z^2}$  on $\OSPTh$   with the following properties.  
 \begin{enumerate}\itemsep=3pt
 \item[{\rm(i)}]   For each $\uvec\in\ri\Uset$,  process 
  \[ \{  \Xw^{\uvec}_x, \, \Bus{\uvec}_{x-\evec_1,x}, \, \Bus{\uvec}_{x-\evec_2,x}, \,\Yw_x:  x\in\Z^2\}  \] 
 is an exponential-$\alpha(\uvec)$ last-passage system as described in Definition \ref{v:d-exp-a}.    With $\P$-probability one, part {\rm(b)} of Definition \ref{v:d-exp-a} holds simultaneously for all $\uvec\in\ri\Uset$. 

 \item[{\rm(ii)}]     There exists a single event $\Omega_0$ of full probability such that for all $\w\in\Omega_0$,   all $x\in\Z^2$ and all $\uvec\prec\vvec$ in $\ri\Uset$ we have the inequalities 
 \be\label{v:853.1} 
\Bus{\uvec}_{x,x+\evec_1}(\w) \ge  \Bus{\vvec}_{x,x+\evec_1}(\w) \quad\text{and}\quad 
\Bus{\uvec}_{x,x+\evec_2}(\w) \le  \Bus{\vvec}_{x,x+\evec_2}(\w).   
\ee
Furthermore,  for all $\w\in\Omega_0$ and $x,y\in\Z^2$,  the function $\uvec\mapsto \Bus{\uvec}_{x,y}(\w)$ is right-continuous with left limits under the ordering \eqref{u-order}. 

 \item[{\rm(iii)}]     For each fixed $\vvec\in\ri\Uset$ there exists an event $\Omega^{(\vvec)}_1$ of full probability such that the following holds: for each $\w\in\Omega^{(\vvec)}_1$  and any  sequence $v_n\in\Z^2$ such that $\abs{v_n}_1\to\infty$ and 
 \be\label{v:853.3}     \lim_{n\to\infty}  \frac{v_n}{\abs{v_n}_1}  = \vvec ,  
 \ee
we have the limits  
 \be\label{v:855}   
 \Bus{\vvec}_{x,y}(\w)  =\lim_{n\to\infty} [ G_{x, v_n}(\w)-G_{y,v_n}(\w)]  \qquad\forall x,y\in\Z^2. \ee
% The LPP process $G_{x,y}$ is now defined by \eqref{v:GY}.  
 Furthermore,  for all $\w\in\Omega^{(\vvec)}_1$ and $x,y\in\Z^2$,  
 \be\label{v:855.7}   \lim_{\uvec\to\vvec}\Bus{\uvec}_{x,y}(\w)=\Bus{\vvec}_{x,y}(\w). \ee 
 
% \item[{\rm(iv)}]  Let $0<\rho<1$.   Suppose that on $\OSP$ there are random variables $(U_x, A_{x-\evec_1,x}, A_{x-\evec_2,x})_{x\in\Z^2}$ such that $\{U_x,  A_{x-\evec_1,x},  A_{x-\evec_2,x}, \Yw_x:  x\in\Z^2\} $ 
% is an exponential-$\rho$ last-passage system as described in Definition \ref{v:d-exp-a}. 
% Then  $U_x= \Xw^\rho_x$, $A_{x-\evec_1,x}=\Bus{\uvec}_{x-\evec_1,x}$ and $A_{x-\evec_2,x}=\Bus{\uvec}_{x-\evec_2,x}$ for all $x$, $\P$-almost surely. 
  \end{enumerate} 
 
 \end{theorem}

 \medskip 

\begin{remark}\label{r:pm}   The process $\uvec\mapsto \Bus{\uvec}$ is globally cadlag (part (ii)) and at each fixed $\vvec$ limit \eqref{v:855.7} holds almost surely.  For each $x,y\in\Z^2$, $\uvec\mapsto \Bus{\uvec}_{x,y}$ is in fact a jump process \cite{fan-sepp-arxiv}.   The cadlag property is  merely a convention.  For certain purposes it can be useful to work with  two processes $\Bus{\uvec}_+(x,y)$   and $\Bus{\uvec}_-(x,y)$ such that     $\uvec\mapsto \Bus{\uvec}_+$ is right-continuous with left limits,    $\uvec\mapsto \Bus{\uvec}_-$ is left-continuous with right limits,  and  $\Bus{\uvec}_+=\Bus{\uvec}_-$ almost surely for a given $\uvec$.   Our results in Theorems \ref{t:main1}--\ref{t:main3}  are almost sure statements  for a fixed $\uvec$, and hence we could use either process $\Bus{\uvec}_+$  or  $\Bus{\uvec}_-$.  
 \hfill$\triangle$\end{remark}

Part (i) of  Theorem \ref{t:buse}   together with \eqref{IJw5.1.7} and \eqref{IJw5.4.7} imply 
\be\label{914}     \Yw_x=\Bus{\uvec}_{x,\,x+\evec_1} \wedge \Bus{\uvec}_{x,\,x+\evec_2}  
\ee
 and 
\be\label{915}     \Xw^{\uvec}_x=\Bus{\uvec}_{x-\evec_1,\,x} \wedge \Bus{\uvec}_{x-\evec_2,\,x} .  
\ee

 From the exponential distributions of $\Bus{\uvec}_{0,\evec_1}$ and $\Bus{\uvec}_{0,\evec_2}$ and the explicit formula \eqref{g} of  the shape function follows 
\be\label{EB8}     \bigl( \, \E[\Bus{\uvec}_{0,\evec_1}]\,, \E[\Bus{\uvec}_{0,\evec_2}]\,\bigr)
=\Bigl(\,  \frac1{\alpha(\uvec)} \,, \frac1{1-\alpha(\uvec)}\,\Bigr) = \nabla\gpp(\uvec). \ee  
This is natural since by  \eqref{v:855}   $\Bus{\uvec}$ can be viewed as the  ``microscopic gradient'' of the passage time. 

The next theorem gives strong uniqueness of the process $\{\Bus{\uvec},\Xw^\uvec\}$.  

  \begin{theorem}\label{t:unique}   Assume \eqref{ass-B} and let $\{\Bus{\uvec},\Xw^\uvec: \uvec\in\ri\Uset\}$ be the process given by Theorem \ref{t:buse}.  Fix $0<\rho<1$. 
Suppose that on $\OSP$ there are random variables $(U_x, A_{x-\evec_1,x}, A_{x-\evec_2,x})_{x\in\Z^2}$ such that $\{U_x,  A_{x-\evec_1,x},  A_{x-\evec_2,x}, \Yw_x:  x\in\Z^2\} $ 
 is an exponential-$\rho$ last-passage system as described in Definition \ref{v:d-exp-a}. 
 Then  $U_x= \Xw^{\uvec(\rho)}_x$, $A_{x-\evec_1,x}=\Bus{\uvec(\rho)}_{x-\evec_1,x}$ and $A_{x-\evec_2,x}=\Bus{\uvec(\rho)}_{x-\evec_2,x}$ for all $x$, $\P$-almost surely. 
%   
%  %%%%%%   REST OF THE UNIQUENESS STATEMENT
%  Assume the setting of Theorem \ref{t:buse} with given i.i.d.\ {\rm Exp$(1)$} weights $\{Y_x\}$ on a probability space $\OSP$.  Fix $0<\rho<1$. 
%  
%  {\rm (i)}      Suppose that, indexed by the nearest-neighbor edges  in the  quadrant with lower left corner at  $u=(u_1,u_2)\in\Z^2$,   there are random variables $(A_{x,\, x+\evec_1}, A_{x,\,x+\evec_2})_{x\,\in\, u+\Z_{\ge0}^2}$ on $\OSP$ with  properties {\rm(a)} and {\rm(b)}:  
%   \begin{enumerate}\itemsep=3pt
% \item[{\rm(a)}]   $A_{x,\, x+\evec_1}\sim$ {\rm Exp}$(1-\rho)$ and  $A_{x,\,x+\evec_2}\sim$ {\rm Exp}$(\rho)$ for all $x\in u+\Z_{\ge0}^2$.   For any $v=(v_1,v_2)\in  u+\Z_{>0}^2$,  the random variables $(A_{v-i\evec_1,\,v-(i-1)\evec_1}, A_{v-j\evec_2,\,v-(j-1)\evec_2})_{1\le i\le v_1-u_1,\, 1\le j\le v_2-u_2 }$ are independent.  
% 
% \item[{\rm(b)}]    For each  $x\in u+\Z_{\ge0}^2$ we have $\P$-almost surely 
%additivity around the  unit square  with lower left corner at $x$: 
% \[    A_{x,\, x+\evec_1}+  A_{x+\evec_1,\,x+\evec_1+\evec_2}= A_{x,\,x+\evec_2}+A_{x+\evec_2,\, x+\evec_1+\evec_2}, 
%% \qquad \forall x\in u+\Z_{\ge0}^2. 
% \]
% and recovery: 
% \[  \Yw_x=A_{x,\, x+\evec_1}\wedge  A_{x,\,x+\evec_2}.\\[4pt]  \] 
% 
%   \end{enumerate} 
% 
% \noindent   
% Then   $A_{x,\, x+\evec_r}=\Bus{\uvec}_{x,\, x+\evec_r}$  for all $x\in u+\Z_{\ge0}^2$ and $r\in\{1,2\}$, $\P$-almost surely. 
%   
%   \smallskip
%   
%{\rm (ii)}  In particular, 
  \end{theorem} 
  
%The second part of assumption (a) above says that the edge variables $A$ are independent on the north and east boundaries of any rectangle $[u,v]$. Part (ii)  about the uniqueness of  an exponential-$\rho$ last-passage system  that includes the given weights $\Yw$ is a consequence of   part (i)  of the theorem because assumptions  (a) and (b) are true for   an exponential-$\rho$ last-passage system.  
 
% \medskip 
 
 \subsection{The idea of the proof of Theorems \ref{t:buse} and \ref{t:unique}}   These theorems are proved in detail in Section 4 of lecture notes \cite{sepp-cgm-18}.  This type of proof was introduced first in the context of the positive-temperature log-gamma polymer in \cite{geor-rass-sepp-yilm-15}.  We   sketch   the main idea.   The essential point  for the message of this paper is that coalescence of geodesics is {\it not used} in the proof, only couplings, monotonicity, and properties of the increment-stationary LPP processes of \eqref{G-a}.    
 
 In  \eqref{v:855} let $\vvec=\uvec(\alpha)$ defined by \eqref{u-a}.
Construct   an exponential-$\lambda$ LPP system in the  quadrant  $x+\Z^2_{\ge0}$, as explained below \eqref{IJw5.4.7}. 
  Use the i.i.d.\ Exp(1) $\eta$-weights of this construction (defined by \eqref{IJw5.1.7})    to define   last-passage times   $G_{x,y}$.   Consider an $\evec_1$-increment $G_{x,v_n}-G_{x+\evec_1,v_n}$  in \eqref{v:855}.   % Property (a) of Definition \ref{v:d-exp-a} applies to this system in the first quadrant.   
  Place the $I$ weights on the north and the $J$ weights on the east boundary of the rectangle $[x,v_n+\evec_1+\evec_2]$.  Use this augmented system to define  last-passage times   $G_{x,v_n+\evec_1+\evec_2}^{\lambda, NE}$, where superscript NE indicates that the boundary weights are on the north and east.    Then, by planar monotonicity (Lemma \ref{lm-til})  and by choosing $\lambda$ suitably,  the  upper bound 
  \[   G_{x,v_n}-G_{x+\evec_1,v_n} \le G_{x,v_n+\evec_1+\evec_2}^{\lambda, NE}-G_{x+\evec_1,v_n+\evec_1+\evec_2}^{\lambda, NE}
  \]  
  holds with high probability for large $n$.  The right-hand increment above can be controlled because it comes from  an increment-stationary LPP process.  % as in \eqref{G-a2}, but now with the boundary conditions in the north and east, rather than south and west. 
 Similar reasoning yields a lower bound
   \[   G_{x,v_n}-G_{x+\evec_1,v_n} \ge G_{x,v_n+\evec_1+\evec_2}^{\rho, NE}-G_{x+\evec_1,v_n+\evec_1+\evec_2}^{\rho, NE}
  \]  
   with a different parameter $\rho$.  After sending $v_n$ to infinity, the  bounds are brought together by letting $\lambda$ and $\rho$ converge to $\alpha$.   
 
  This establishes the almost sure limit   \eqref{v:855}   for a countable dense set of  directions $\vvec$.  Properties of the resulting processes $\Bus{\vvec}$ are derived   from monotonicity and the  increment-stationary LPP processes.  The construction of the full process $\{\Bus{\uvec}:\uvec\in\ri\Uset\}$ is completed by taking right limits as $\vvec\searrow\uvec$ to get  cadlag paths in the parameter $\uvec$.   This proves Theorem \ref{t:buse}.
 
 To prove the uniqueness in Theorem \ref{t:unique}, the reasoning above is repeated: this time increment  variables $A_{x-\evec_k,x}$ are given, and planar monotonicity is used to sandwich them between Busemann limits from Theorem \ref{t:buse}.  

%\medskip

% \section{Preliminaries}  \label{s:prel} 
% 
%This section   collects well-known CGM results needed  in the sequel.  Most of them are proved  in   the lecture notes \cite{sepp-cgm-18}.  
 
%\subsection{Shape theorem}  \label{s:sh-g}

\subsection{Midpoint problem} 
We quote one more result from \cite{sepp-cgm-18}  that is a corollary of the Busemann limits.    We use this fact in the proof of Theorem \ref{t:bg105} below to show the nonexistence of   bi-infinite $\Bus{\uvec}$-geodesics. 
Let  $\pi^{x,y}$ %  $\pi^{x,y}_\dbullet$ 
denote the (almost surely unique)    geodesic  for $\Gpp_{x,y}$ defined by \eqref{v:GY7}.  

% The {\it midpoint problem} asks whether the probability that the geodesic goes through some  particular point roughly midway between $x$ and $y$ converges to zero,  as the distance  between $x$ and $y$ grows.    We prove this in the next theorem for the case where $y-x$ has a limiting direction.    

%In undirected first-passage percolation the midpoint problem was solved under a differentiability assumption on the limit shape by \cite{damr-hans-17} and without any such unverifiable assumptions by \cite{ahlb-hoff}. 

%%%%%    earlier definition of levels.  Not needed now. 
%For $m\in\Z_{\ge0}$ let 
%\[  \bL_{m}=\{x\in\Z_{\ge0}^2: \abs{x}_1=m\}=\{(j, m-j): 0\le j\le m\}\]
% denote the set  of lattice points   on level $m$ in the first quadrant. 
%Fix a parameter $\rho\in(0,1)$ and for $n\in\Z_{>0}$ let $v_n\in\bL_n$ be a sequence of lattice points that escapes in the characteristic direction of $\rho$:    $v_n/n\to\xi(\rho)$.  
%Then fix $t\in(0,1)$ to select an intermediate level.  Let $z_n\in\bL_{\fl{nt}}$ be a sequence of lattice points  on successive  levels $\fl{nt}$ such that $z_n/n\to t\xi(\rho)$.    
%%The next theorem shows  that the probability that the maximizing path (or geodesic) for $\Gpp_{0,v_n}$ goes through $z_n$ converges to zero. 
%

\begin{theorem} \label{th:midpt}    Assume \eqref{ass-B} and fix $\uvec\in\ri\Uset$.   Let $u_n\le z_n\le v_n$ be three  sequences on $\Z^2$ that satisfy the following conditions:  $u_n$ and $v_n$ can be random but   $z_n$ is not  {\rm(}that is, $u_n$ and $v_n$ can be measurable  functions of $\w$ but $z_n$ does not depend on $\w${\rm)},   $\abs{v_n-z_n}_1\to\infty$,  $\abs{z_n-u_n}_1\to\infty$,  and 
\[ \lim_{n\to\infty} \frac{v_n-z_n}{\abs{v_n-z_n}_1} = 
%\uvec  \quad\text{and}\quad 
\lim_{n\to\infty} \frac{z_n-u_n}{\abs{z_n-u_n}_1} = \uvec. \]
Then  $\ddd\lim_{n\to\infty} \P\{ z_n\in \pi^{u_n,v_n} \}=0$. 

\end{theorem} 

This theorem  is proved for deterministic  $u_n, v_n$ in lecture notes \cite{sepp-cgm-18}  as Theorem 4.12 on p.~174.  The same argument proves the version above for random $u_n, v_n$ and appears in the arXiv version of \cite{sepp-cgm-18}.   The proof proceeds by expressing the condition $z_n\in \pi^{u_n,v_n}$ in terms of  increments of   $G_{x,y}$ and then taking the Busemann limits \eqref{v:855}.

%\section{Busemann geodesics, competition interface, and proofs of the main theorems} 
\section{Busemann geodesics  and proofs of the main theorems}  \label{s:bg} 

\subsection{Busemann geodesics}  %\label{s:bg}

Let   $\{\Bus{\uvec}: \uvec\in\ri\Uset\} $ be the covariant integrable cocycles constructed in Theorem \ref{t:buse}.  
%  In particular, each process 
%  \[ \{  \Xw^{\uvec}_x, \, \Bus{\uvec}_{x-\evec_1,x}, \, \Bus{\uvec}_{x-\evec_2,x}, \,\Yw_x:  x\in\Z^2\}  \] 
%is an   exponential-$\alpha$ last-passage system.  
 We   write  interchangeably $\Bus{\uvec}(x,y,\w)=\Bus{\uvec}_{x,y}(\w)$.
For each direction $\uvec\in\ri\Uset$ and initial point $x\in\Z^2$ construct a semi-infinite  random  up-right  lattice path 
$\bgeodu{x}(\w)%=\bgeodu{x}(\w)
=\{\bgeodu{x}_k(\w)\}_{k\in\Z_{\ge0}}$ by following minimal increments of $\Bus{\uvec}$: 
\be\label{bg6} \begin{aligned}  
\bgeodu{x}_0(\w)&=x, \quad \text{and for $k\ge 0$}  \\[4pt] 
\bgeodu{x}_{k+1}(\w)&=\begin{cases}   \bgeodu{x}_{k}(\w) + \evec_1, &\text{if } \ \Bus{\uvec}_{\bgeodu{x}_{k},\,\bgeodu{x}_{k} + \evec_1} (\w)\le  \Bus{\uvec}_{\bgeodu{x}_{k},\,\bgeodu{x}_{k} + \evec_2}(\w) 
%\Bus{\uvec}(\bgeodu{x}_{k},\bgeodu{x}_{k} + \evec_1, \w)\le  \Bus{\uvec}(\bgeodu{x}_{k},\bgeodu{x}_{k} + \evec_2,\w) 
\\[5pt]  
 \bgeodu{x}_{k}(\w) + \evec_2, &\text{if } \   \Bus{\uvec}_{\bgeodu{x}_{k},\,\bgeodu{x}_{k} + \evec_2}(\w) <  \Bus{\uvec}_{\bgeodu{x}_{k},\,\bgeodu{x}_{k} + \evec_1}(\w) . 
%\Bus{\uvec}(\bgeodu{x}_{k},\bgeodu{x}_{k} + \evec_2, \w) <  \Bus{\uvec}(\bgeodu{x}_{k},\bgeodu{x}_{k} + \evec_1, \w) . 
\end{cases} 
\end{aligned} \ee
The tie-breaking rule in favor of $\evec_1$  is a convention we follow henceforth.   For a given $\uvec$ 
the case of equality on the right-hand side of the two-case formula happens with probability zero because $\Bus{\uvec}_{x,x+\evec_1}$ and $\Bus{\uvec}_{x,x+\evec_2}$ are independent exponential random variables.    Pictorially, to each point $z$  attach an arrow that points from $z$ to $\bgeodu{z}_1$.  The path  $\bgeodu{x}$ is constructed by starting at $x$ and  following the arrows.   By \eqref{914}, 
\be\label{bg8}   \Yw_{\bgeodu{x}_{k}}=\Bus{\uvec}(\bgeodu{x}_{k},\bgeodu{x}_{k+1}) 
\quad\text{for $k\ge 0$.}  \ee
 We shall call $\bgeodu{x}$ the {\it $\Bus{\uvec}$-geodesic} from $x$.  This term is justified by the next lemma.    Since the processes $\Bus{\uvec}$ arise as Busemann functions, we can also call these geodesics {\it Busemann geodesics}.

\begin{lemma}\label{Bgeod-lm}  $ $ 

\begin{enumerate}[{\rm(i)}] \itemsep=4pt
\item $\bgeodu{x}$ is a semi-infinite geodesic for the LPP process \eqref{v:GY7}.  For all $0\le m<n$,  
\be\label{bgeod93} G({\bgeodu{x}_m,\bgeodu{x}_n}) =   \Bus{\uvec}(\bgeodu{x}_m,\bgeodu{x}_n) + \Yw_{\bgeodu{x}_n}. \ee 

\item There exists an event $\Omega_2$ such that $\P(\Omega_2)=1$ and for all $\w\in\Omega_2$ the following properties hold $\forall \uvec, \vvec\in\ri\Uset$.    If $\uvec\prec\vvec$, then $\bgeod{\vvec}{x}$ stays always {\rm(}weakly{\rm)} to the right and below $\bgeodu{x}$.     Furthermore,   geodesic  $\bgeodu{x}$ is $\uvec$-directed: 
 \be\label{bgeod98} \lim_{n\to\infty} \frac{\bgeodu{x}_n}n =\uvec 
\qquad  \forall x\in\Z^2. \ee

\item   For each fixed $\vvec\in\ri\Uset$ there exists an event $\Omega^{(\vvec)}_3$  such that $\P(\Omega^{(\vvec)}_3)=1$ and  the following properties hold for each $\w\in\Omega^{(\vvec)}_3$ and $x\in\Z^2${\rm:}  $\forall k\in\Z_{\ge0}$,   $\bgeodu{x}_k\to \bgeod{\vvec}{x}_k$ as $\uvec\to\vvec$ in $\ri\Uset$, and   furthermore,  $\bgeod{\vvec}{x}$ is the unique semi-infinite $\vvec$-directed geodesic out of $x$.   In particular, the geodesic tree $\cT_\vvec$ defined by \eqref{T1} can be expressed as 
\be\label{T3}     \cT_\vvec=\bigcup_{x\in\Z^2} \bgeod{\vvec}{x}  \ee  
where again geodesics are regarded as collections of edges. 
\end{enumerate} 
\end{lemma} 

\begin{proof}  Part (i).  Let $x_{0,n}$ be any path from  $x_0=\bgeodu{x}_0=x$ to $x_n=\bgeodu{x}_n$.  By \eqref{914} and \eqref{bg8}, 
\begin{align*}
\sum_{k=0}^{n}\Yw_{x_k} &\le \sum_{k=0}^{n-1}  \Bus{\uvec}(x_k,x_{k+1}) + \Yw_{x_n} 
= \Bus{\uvec}(x_0,x_n) + \Yw_{x_n} = \Bus{\uvec}(\bgeodu{x}_0,\bgeodu{x}_n) + \Yw_{\bgeodu{x}_n} \\ 
&=\sum_{k=0}^{n-1}  \Bus{\uvec}(\bgeodu{x}_k,\bgeodu{x}_{k+1}) + \Yw_{\bgeodu{x}_n}
=\sum_{k=0}^{n}\Yw_{\bgeodu{x}_k}. 
\end{align*}
Thus for any $n$,  the segment $\bgeodu{x}_{0,n}$ is a geodesic between its endpoints. 

\smallskip

Part (ii).   The ordering of Busemann geodesics follows from the monotonicity \eqref{v:853.1}  of the Busemann functions.

For  the limit \eqref{bgeod98} consider first   fixed $\uvec\in\ri\Uset$.   Recall the mean vector of $\Bus{\uvec}$ from \eqref{EB8}.    The cocycle  ergodic theorem (Theorem \ref{cc-ergthm} in the Appendix) applies to the mean-zero cocycle  \[  F(\w,x,y)=-\Bus{\uvec}(x,y,\w)+\nabla \gpp(\uvec)\cdot (y-x)\]  by virtue of the bound 
$  F(\w,0,\evec_i)  \le -Y_0 + C   $ 
that comes from   \eqref{914}.  
By translation-invariance, if  \eqref{bgeod98} is proved for $x=0$ it follows for all   $x$. Since $\bgeodu{0}\subset\Z_{\ge0}^2$,  $\gpp({\bgeodu{0}_n})$ is defined for the shape function $\gpp$ in \eqref{g}.  
 Then 
by the homogeneity of $\gpp$,   \eqref{bgeod93},  \eqref{sh89},  and Theorem \ref{cc-ergthm}, 
\begin{align*}
&\gpp\biggl(\frac{\bgeodu{0}_n}n \biggr) -  \nabla \gpp(\uvec)\cdot  \frac{\bgeodu{0}_n}n \\
&\quad =\frac1n\bigl[ \gpp({\bgeodu{0}_n}) -  G({0 ,\bgeodu{0}_n}) \bigr]  + \frac1n\bigl[ \Bus{\uvec}(0,\bgeodu{0}_n)  - \nabla \gpp(\uvec)\cdot  {\bgeodu{0}_n} \bigr] + \frac{\Yw_{\bgeodu{0}_n}}n\\[4pt] 
&\quad \longrightarrow 0 \qquad \text{almost surely as $n\to\infty$.} 
\end{align*}

All the limit points of $\bgeodu{0}_n/n$ lie on $\Uset$. 
As a differentiable, concave and homogeneous function, $\gpp$ satisfies $\gpp(\xi)=\nabla\gpp(\xi)\cdot\xi$ for all $\xi\in\R_{>0}^2$.     Since $\gpp$ is strictly concave  on $\Uset$,   for every $\delta>0$ there exists $\e>0$ such that  
\be\label{bg-gpp7} \gpp(\vvec)\le \nabla\gpp(\uvec)\cdot\vvec-\e  \qquad\text{for $\vvec\in\Uset$ such that  $\abs{\vvec-\uvec}\ge\delta$. }\ee
 Thus the limit above forces $\bgeodu{0}_n/n\to\uvec$ almost surely.  

Let  $\Omega_2$ be  the event on which   limit  \eqref{bgeod98} happens for a countable  dense set of directions  $\uvec\in\ri\Uset$ and all $x\in\Z^2$.  The limit extends simultaneously to all $\uvec\in\ri\Uset$ on the event   $\Omega_2$ by virtue of the ordering of the geodesics $\bgeodu{x}$.

\smallskip

Part (iii).   Let  $\Omega^{(\vvec)}_3$ be the event on which limits  \eqref{v:855.7}  hold,  uniqueness of finite geodesics holds,   equality on the right-hand side of \eqref{bg6} does not happen  for the fixed $\vvec$, and part (ii) above holds.   On this event 
 $\bgeodu{x}_k\to \bgeod{\vvec}{x}_k$ as $\uvec\to\vvec$  because, inductively in  $k$, \eqref{bg6} chooses the same step for all $\uvec$ close enough to $\vvec$ by virtue of  \eqref{v:855.7}.   
 
 Let $\pi=(\pi_i)_{i\in\Z_{\ge0}}$ be a $\vvec$-directed semi-infinite geodesic from $\pi_0=x$.   Let $\uvec\prec\vvec\prec\wvec$ in $\ri\Uset$.   By the directedness \eqref{bgeod98},  after some (random but finite) number of steps $\pi$ remains strictly between $\bgeodu{x}$ and $\bgeod{\wvec}{x}$.   Then it follows that $\pi$ remains for all time weakly between $\bgeodu{x}$ and $\bgeod{\wvec}{x}$.  
  For if $\pi$ ever went strictly to the left of  $\bgeodu{x}$, it would have to eventually intersect  $\bgeodu{x}$ at some later point $\pi_m=\bgeodu{x}_m$.    Then there would   be two distinct geodesics $\pi_{0,m}$ and  $\bgeodu{x}_{0,m}$ from $x$ to $\pi_m$, in violation of the uniqueness of finite geodesics.   Similarly $\pi$ cannot go   strictly to the right of  $\bgeod{\wvec}{x}$.  
  
  Letting  $\uvec\to\vvec$ and $\wvec\to\vvec$ shows that $\pi$ must  coincide with  $\bgeod{\vvec}{x}$.
  \end{proof}

\begin{proof}[Proof of Theorem \ref{t:main-1}] 
Part (i).   Let $\Omega_4$ be the full probability event on which finite geodesics are unique and limits \eqref{bgeod98} hold for all $x\in\Z^2$ and all $\uvec\in\ri\Uset$. Fix $\w\in\Omega_4$.    Let $x_\dbullet=(x_n)_{n\ge n_0}$ be a semi-infinite geodesic at this sample point $\w$.  We can assume it indexed so that $x_n\cdot(\evec_1+\evec_2)=n$.    Suppose 
\be\label{bg701}    \underline u_1= \varliminf   \frac{x_n\cdot \evec_1}n  <   \varlimsup_{n\to\infty}    \frac{x_n\cdot \evec_1}n  =  \bar u_1.  \ee
Then necessarily $0\le \underline u_1< \bar u_1 \le1$.  
 Pick a vector $\uvec\in\ri\Uset$ between $\underline\uvec=(\underline u_1, 1-\underline u_1)$ and $\bar\uvec=(\bar u_1, 1-\bar u_1)$. 
Then infinitely often $x_\dbullet$ is strictly to the left of, strictly to the right of, and  crosses $\bgeodu{x_{n_0}}$.  This violates the uniqueness of finite geodesics.  Consequently \eqref{bg701} cannot happen on $\Omega_4$ and hence all semi-infinite geodesics have a direction.

% Next we show that $\{x_k=k\evec_1\}_{k\in\Z_{\ge0}}$ is the only semi-infinite geodesic from the origin such that $x_n/n\to\evec_1$.  This will conclude the proof of part (i).  
 
%   \begin{itemize}
%  \item  Intersecting $\Omega_4$  with the (countably many full probability) events $\Omega^{(\wvec_k)}_1$ from Theorem \ref{t:buse}(iii) so that the Busemann limit \eqref{v:855} holds for $\vvec=\wvec_k$ for each $k$.   
%  \item Assume that 
 % \end{itemize} 

\medskip

Part (ii).   We prove the case $\evec_1$ for $x=0$.    Fix a sequence $\wvec_1\prec\wvec_2\prec\dotsm\prec\wvec_k\prec\dotsm$ in $\ri\Uset$ such that  $\wvec_k\to\evec_1$. By Theorem \ref{t:buse}, $\Bus{\wvec_k}_{0, \evec_2}\sim$ Exp$(1-\alpha(\wvec_k))$.  Since %$\wvec_k\to\evec_1$, 
  $1-\alpha(\wvec_k)\to 0$,  
  \be\label{bg704} \Bus{\wvec_k}_{0, \evec_2}\to\infty\qquad\text{almost surely as $k\to\infty$}
  \ee
   by the monotonicity \eqref{v:853.1}.  
     While retaining $\P(\Omega_4)=1$, modify the event $\Omega_4$ so that \eqref{bg704} holds on $\Omega_4$, and further  intersect it with the (countably many full probability) events $\Omega^{(\wvec_k)}_1$ from Theorem \ref{t:buse}(iii).  Now the Busemann limit \eqref{v:855} holds on $\Omega_4$ for $\vvec=\wvec_k$ for each $k$.   
     
 Fix $\w\in\Omega_4$.  Suppose that at this $\w$ there is a   semi-infinite geodesic $\pi=\{\pi_n\}_{n\in\Z_{\ge0}}$  such that $\pi_0=0$, $\pi_\ell=(\ell-1,1)$ for some $\ell\ge 1$, and $\varliminf_{n\to\infty} n^{-1} \pi_n\cdot \evec_2=0$.  We derive a contradiction from this. 

By connecting $\evec_2=(0,1)$ to the point $\pi_\ell=(\ell-1,1)$ (now fixed for the present)  with a horizontal path, we get the lower bound 
\[  \Gpp_{\evec_2, \pi_n}\ge  \sum_{i=0}^{\ell-1}\w_{(i,1)} + \Gpp_{\pi_{\ell+1}, \pi_n} 
\qquad \text{for } n>\ell. \]
That $\pi$ is a geodesic from $\pi_0=0$ implies  $\Gpp_{0, \pi_n}= \Gpp_{0, \pi_\ell} + \Gpp_{\pi_{\ell+1}, \pi_n} $ for $n>\ell$. 
Thus 
\be\label{406}  \Gpp_{0, \pi_n} -  \Gpp_{\evec_2, \pi_n}\le \Gpp_{0, \pi_\ell}-  \sum_{i=0}^{\ell-1}\w_{(i,1)}
 \qquad \text{for all } n>\ell. \ee
 For each $k$, fix a sequence $\{w_{n,k}\}_{n\ge 0}$ in $\Z^2_{\ge 0}$ such that $\abs{w_{n,k}}_1=n$ and $\lim_{n\to\infty} n^{-1}w_{n,k}=\wvec_k$. 
 By the assumptions $\varliminf n^{-1} \pi_n\cdot \evec_2=0$ and $\wvec_k\in\ri\Uset$, and  by Lemma \ref{lm-til}, there are infinitely many indices $n$ such that 
 \[ \Gpp_{0, \pi_n} -  \Gpp_{\evec_2, \pi_n}\ge \Gpp_{0, w_{n,k}} -  \Gpp_{\evec_2, w_{n,k}}. 
  \]  
Hence by the Busemann limit \eqref{v:855},
 \[  \varlimsup_{n\to\infty} [\Gpp_{0, \pi_n} -  \Gpp_{\evec_2, \pi_n} ] \ge \Bus{\wvec_k}_{0, \evec_2}. 
  \]  
Limit \eqref{bg704} now contradicts \eqref{406} because the right-hand side of \eqref{406} is fixed and finite. 

\medskip

Part (iii).  The family $\{\bgeodu{x}: \uvec\in\ri\Uset, x\in\Z^2\}$ gives a $\uvec$-directed semi-infinite geodesic for each $\uvec\in\ri\Uset$ and each starting point $x$.   A semi-infinite  geodesic in direction $\evec_r$ from $x$ is defined trivially by $x_k=x+k\evec_r$ for $k\ge 0$. 
\end{proof}

  \begin{proof}[Proof of \eqref{a1}] 
By part (iii) of Lemma \ref{Bgeod-lm} and by \eqref{bg6}, 
\[   \P\{ \ageodu{x}_1=x+\evec_1\} =    \P\{ \bgeodu{x}_1=x+\evec_1\} =  \P\{ \Bus{\uvec}_{x,x+\evec_1} \le  \Bus{\uvec}_{x,x+\evec_2} \} = \alpha. 
\]
The last equality is due to the fact that $\Bus{\uvec}_{x,x+\evec_1}$  and $\Bus{\uvec}_{x,x+\evec_2}$ are independent exponential random variables with rates $\alpha$ and $1-\alpha$, respectively.   This comes from part (i) of Theorem \ref{t:buse} because $(x+\evec_2, x,x+\evec_1)$ is a segment of a down-right path. 
\end{proof}

  \begin{remark}\label{r:pm2}   If two separate Busemann processes   $\Bus{\uvec}_+$   and $\Bus{\uvec}_-$    are constructed as indicated in Remark \ref{r:pm},  then two Busemann geodesics    $\bgeodu{x,+}$  and $\bgeodu{x,-}$ would be constructed by \eqref{bg6}.    Lemma \ref{Bgeod-lm} would hold for both families.  Furthermore,   $\bgeodu{x,+}$ would always stay weakly to the right and below $\bgeodu{x,-}$. 
\hbox{ }  \hfill$\triangle$\end{remark}

In view of Lemma \ref{Bgeod-lm}(iii),  to complete the proof of Theorem \ref{t:main1}, it suffices to prove that, $\P$-almost surely  for a fixed $\uvec\in\ri\Uset$,  geodesics $\bgeodu{x}$ and $\bgeodu{y}$ coalesce and that there is no bi-infinite $\Bus{\uvec}$-geodesic.   To achieve  this we introduce dual geodesics and along the way  prove   Theorem \ref{t:main2}.  

%\medskip 

\subsection{South-west and dual geodesics} 

Define {\it south-west $\Bus{\uvec}$-geodesics}   
$\swbgeodu{x}(\w)$ by following minimal south-west  increments of $\Bus{\uvec}$: 
\be\label{bg16} \begin{aligned}  
\swbgeodu{x}_0(\w)&=x, \quad \text{and for $k\ge 0$}  \\[4pt] 
\swbgeodu{x}_{k+1}(\w)&=\begin{cases}   \swbgeodu{x}_{k}(\w) - \evec_1, &\text{if } \ \Bus{\uvec}_{\swbgeodu{x}_{k}-\evec_1,\,\swbgeodu{x}_{k}}( \w)\le   \Bus{\uvec}_{\swbgeodu{x}_{k}-\evec_2,\,\swbgeodu{x}_{k}}(\w) 
%\Bus{\uvec}(\swbgeodu{x}_{k}-\evec_1,\swbgeodu{x}_{k} , \w)\le   \Bus{\uvec}(\swbgeodu{x}_{k}-\evec_2,\swbgeodu{x}_{k} ,\w)
\\[6pt]  
 \swbgeodu{x}_{k}(\w) - \evec_2, &\text{if } \   \Bus{\uvec}_{\swbgeodu{x}_{k}-\evec_2,\,\swbgeodu{x}_{k}}(\w) <   \Bus{\uvec}_{\swbgeodu{x}_{k}-\evec_1,\,\swbgeodu{x}_{k}}(\w) . 
 %\swbgeodu{x}_{k}(\w) - \evec_2, &\text{if } \   \Bus{\uvec}(\swbgeodu{x}_{k}-\evec_2,\swbgeodu{x}_{k} , \w) <   \Bus{\uvec}(\swbgeodu{x}_{k}-\evec_1,\swbgeodu{x}_{k} , \w) . 
\end{cases} 
\end{aligned} \ee
 By \eqref{915}, 
\be\label{bg18}   \Xw^\uvec_{\swbgeodu{x}_{k}}=\Bus{\uvec}(\swbgeodu{x}_{k+1},\swbgeodu{x}_{k}) 
\quad\text{for $k\ge 0$.}  \ee
Define an LPP  process in terms of the weights $\Xw^{\uvec}$:  
	\be\label{v:GX7}  
	  \Gpp^{\Xw^{\uvec}}_{x,y}=\Gpp^{\Xw^{\uvec}}(x,y)=\max_{x_{\brbullet}\,\in\,\Pi_{x,y}}\sum_{k=0}^{\abs{y-x}_1}\Xw^{\uvec}_{x_k}  \qquad \text{for } x\le y\text{ on } \Z^2.   
	\ee
We think of this   LPP process as  pointing  down and left,  but do not  alter the ordering $x\le y$ in the notation $\Gpp^{\Xw^{\uvec}}_{x,y}$.  
%  Lemma \ref{Bgeod-lm} becomes now the following.  

\begin{lemma}\label{Bgeod-lmX}   Fix $\uvec\in\ri\Uset$.   

\begin{enumerate}[{\rm(i)}] \itemsep=4pt
\item
$\swbgeodu{x}$ is a semi-infinite down-left geodesic for  LPP process  $\Gpp^{\Xw^{\uvec}}$  defined by \eqref{v:GX7}. % that uses  weights $\Xw^{\uvec}$.   
For all $0\le m<n$,  
\be\label{bgeod93X} G^{\Xw^{\uvec}}(\swbgeodu{x}_n,\swbgeodu{x}_m) =   \Bus{\uvec}(\swbgeodu{x}_n,\swbgeodu{x}_m) + \Xw^{\uvec}_{\swbgeodu{x}_n}. \ee 

\item  We have the $\P$-almost sure direction 
\be\label{bsw98} \lim_{n\to\infty} \frac{\swbgeodu{x}_n}n =-\uvec 
\qquad  \forall x\in\Z^2. \ee

\item  $\Bus{\uvec}$ is the Busemann function for  LPP process  $\Gpp^{\Xw^{\uvec}}$ in direction $-\uvec$. Precisely,   on the event $\Omega^{(\uvec)}_1$  of Theorem \ref{t:buse}{\rm(iii)} and for any  sequence $v_n\in\Z^2$ such that $\abs{v_n}_1\to\infty$ and 
${v_n}/{\abs{v_n}_1} \to -\uvec$,  
\be\label{Bsw}   
 \Bus{\uvec}_{x,y}  = \lim_{n\to\infty}  [\,\Gpp^{\Xw^{\uvec}}_{v_n ,y} - \Gpp^{\Xw^{\uvec}}_{v_n,x} \,] \qquad\forall x,y\in\Z^2. 
\ee

\end{enumerate}  \end{lemma} 

\begin{proof}   Part (i) is proved as in Lemma \ref{Bgeod-lm}, by utilizing \eqref{915} and \eqref{bg18}. 

Define  a process $(\wt X, \wt B^\uvec, \wt Y^\uvec)$  by setting  
\be\label{Ytil4}    \text{$\wt X_x=\Yw_{-x}$,  
$\wt B^\uvec_{x,y}=\Bus{\uvec}_{-y,-x}$,  and 
$\wt Y^\uvec_x=\Xw^{\uvec}_{-x}$  $\ \ \forall x,y\in\Z^2$.   }\ee  
%\begin{align*}
%\wt X_x&=\Yw_{-x} \\
%\wt B_{x,y}&=\Bus{\uvec}_{-y,-x} \\
%\wt Y_x&=\Xw^{\uvec}_{-x} . 
%\end{align*}
Properties of  $(\Xw^\uvec, \Bus{\uvec}, \Yw)$  given  in Theorem \ref{t:buse} imply that  $\{\wt X_x, \wt B^\uvec_{x-\evec_1,x}, \wt B^\uvec_{x-\evec_2,x}, \wt Y^\uvec_x\}_{x\in\Z^2}$ is an exponential-$\alpha(\uvec)$ LPP system.   By Theorem \ref{t:unique},  $\wt B^\uvec$ is the Busemann function in direction $\uvec$ of the LPP process  $ \Gpp^{\wt Y^{\uvec}}$ defined by  \eqref{v:GY7} but with weights $\wt Y^\uvec$.   

For part (ii), apply definition \eqref{bg6} to $\wt B^\uvec$ and compare the outcome with \eqref{bg16} to conclude that $-\swbgeodu{x}$ is the $\wt B^\uvec$ geodesic that starts at $-x$.   Limit \eqref{bgeod98} applied to the LPP process $ \Gpp^{\wt Y^{\uvec}}$ gives \eqref{bsw98}.  

Part (iii) follows from 
\[ %\begin{align*}
 \lim_{n\to\infty}  [\,\Gpp^{\Xw^{\uvec}}_{v_n ,y} - \Gpp^{\Xw^{\uvec}}_{v_n,x} \,] 
 =   \lim_{n\to\infty}  [\,\Gpp^{\wt Y^{\uvec}}_{-y, -v_n } - \Gpp^{\wt Y^{\uvec}}_{-x, -v_n} \,]
 =  \wt B^\uvec_{-y, -x}= \Bus{\uvec}_{x,y}. 
\qedhere \] %\end{align*} 
\end{proof} 

\medskip 

 Define {\it dual $\Bus{\uvec}$-geodesics}  $ \dbgeodu{z}$ on the dual lattice $\Z^{2*}$ by shifting south-west geodesics by $\duale=(\tfrac12,\tfrac12)$: 
\be\label{bg22}   \dbgeodu{z}_{k} =  \swbgeodu{z+\duale}_{k}-\duale \qquad\text{  for $z\in\Z^{2*}$ and $k\ge0$}. \ee

\begin{lemma}\label{dgeod-lm1}   Fix $\uvec\in\ri\Uset$.    Then an edge $e$ lies on some geodesic $\bgeodu{x}$ if and only if its dual edge $e^*$ does not lie on any dual geodesic $ \dbgeodu{z}$.    
In particular,  the family $\{\bgeodu{x}: x\in\Z^2\}$ of $\Bus{\uvec}$-geodesics and the family $\{\dbgeodu{z}: z\in\Z^{2*}\}$ of dual $\Bus{\uvec}$-geodesics never cross each other.  
\end{lemma} 

\begin{proof}    We need to check that, for $x\in\Z^2$,  $\bgeodu{x}_1=x+\evec_1$ if and only if $\dbgeodu{x+\duale}_1=x+\duale-\evec_1$. 
\begin{align*}
\dbgeodu{x+\duale}_1=x+\duale-\evec_1  \ &\Longleftrightarrow \   \swbgeodu{x+\evec_1+\evec_2}_1=x+\evec_2  \\
&\Longleftrightarrow \   \Bus{\uvec}_{x+\evec_2,\,x+\evec_1+\evec_2} \le  \Bus{\uvec}_{x+\evec_1,\,x+\evec_1+\evec_2}  \\
&\Longleftrightarrow \   \Bus{\uvec}_{x,\,x+\evec_1} \le  \Bus{\uvec}_{x,\,x+\evec_2}  \\
&\Longleftrightarrow \   \bgeodu{x}_1=x+\evec_1. 
\end{align*}
The third equivalence used additivity. 
A similar argument shows that  $\bgeodu{x}_1=x+\evec_2$ if and only if $\dbgeodu{x+\duale}_1=x+\duale-\evec_2$. 
\end{proof}  

\begin{lemma}\label{dgeod-lm2}   Fix $\uvec\in\ri\Uset$.  The process of arrows  $\{\bgeodu{x}_1-x\}_{x\in\Z^2}$  is equal in distribution to the process $\{-x-\duale-\dbgeodu{-x-\duale}_1\}_{x\in\Z^2}$ of reversed dual arrows reflected across the origin.  

\end{lemma}

\begin{proof}  Utilize again the process defined in \eqref{Ytil4}.  As observed,  $\wt B^\uvec$ is the Busemann function in direction $\uvec$ of the LPP process \eqref{v:GY7} with weights $\wt Y^\uvec$. 
 In particular then   processes  $\wt B^\uvec$ and $\Bus{\uvec}$ are equal in distribution.   Distributional equality  $\{-x-\swbgeodu{-x}_1\}_{x\in\Z^2}\deq\{\bgeodu{x}_1-x\}_{x\in\Z^2} $ follows from these equivalences:
\be\label{bg26} \begin{aligned}
-x-\swbgeodu{-x}_1=\evec_1\ &\Longleftrightarrow \   \swbgeodu{-x}_1=-x-\evec_1 \ \Longleftrightarrow \  \Bus{\uvec}_{-x-\evec_1,-x}\le \Bus{\uvec}_{-x-\evec_2,-x} \\
&\Longleftrightarrow \   \wt B^\uvec_{x,x+\evec_1}\le \wt B^\uvec_{x,x+\evec_2} 
\end{aligned}\ee 
and 
\begin{align*}
\bgeodu{x}_1-x =\evec_1\ &\Longleftrightarrow \  \Bus{\uvec}_{x,x+\evec_1}\le \Bus{\uvec}_{x,x+\evec_2}.
\end{align*}
The claim of the lemma follows from   $-x-\duale-\dbgeodu{-x-\duale}_1=-x-\swbgeodu{-x}_1$. 
%$-z-\dbgeodu{-z}_1=-z-\duale-\swbgeodu{-z+\duale}_1$. \
%
%
%%%%%%   REFLECTION IDEA THAT DIDN'T WORK   %%%%
% Define reflection of the weights by $(R\Yw)_x=\Yw_{-x}$.  Then for $a\le b$ in $\Z^2$ with $n=\abs{b-a}_1$, define a bijection between  $x_\bbullet\in\Pi_{a,b}$ and $y_\bbullet\in\Pi_{-b,-a}$ by $y_i=-x_{n-i}$. With these conventions,  
%\begin{align*}
%\Gpp_{a,b}\circ R= \max_{x_{\brbullet}\,\in\,\Pi_{a,b}}\sum_{i=0}^{n}\Yw_{-x_i}
%=\max_{y_{\brbullet}\,\in\,\Pi_{-b,-a}}\sum_{i=0}^{n}\Yw_{y_i}=\Gpp_{-b,-a}. 
%\end{align*}
%
%Define 
%\[  \swBa_{x,y} = \Bus{\uvec}_{-y,-x}.\]
%Take a sequence $v_n$ that realizes the limit in \eqref{v:855}. 
%\begin{align*}
%\swBa_{x,y} = \Bus{\uvec}_{-y,-x}= \lim_{n\to\infty} [G_{-y,v_n}-G_{-x,v_n}] 
%= \lim_{n\to\infty} [G_{-v_n, y}-G_{-v_n,x}]\circ R
%\end{align*}
%
%\begin{align*}
% \Bus{\uvec}_{-y,-x}\circ R= \lim_{n\to\infty} [G_{-y,v_n}-G_{-x,v_n}] \circ R
%= \lim_{n\to\infty} [G_{-v_n, y}-G_{-v_n,x}]
%\end{align*}
\end{proof} 

The message of the last two lemmas is that  the up-right directed $\Bus{\uvec}$-geodesics $\{\bgeodu{x}: x\in\Z^2\}$  and the down-left directed dual $\Bus{\uvec}$-geodesics  $\{\dbgeodu{z}: z\in\Z^{2*}\}$  never cross each other but are equal in distribution, modulo a shift by $\duale$ and a  lattice reflection across the origin.

\subsection{Coalescence and the bi-infinite geodesic} 

The {\it backward $\Bus{\uvec}$-cluster}  $\cC^\uvec(x)$  at $x$  consists of those points $y$ whose $\Bus{\uvec}$-geodesic goes through $x$: 
\[   \cC^\uvec(x)=\{ y\in x+\Z^2_{\le0}:  \bgeodu{y}_{\abs{x-y}_1}=x\}  . \] 
A bi-infinite up-right nearest-neighbor  path $\{x_k\}_{k\in\Z}$ on $\Z^2$  is a {\it bi-infinite $\Bus{\uvec}$-geodesic} if $\bgeodu{x_k}_{\ell-k}=x_\ell$ for all indices $k<\ell$ in $\Z$. 
If   two $\Bus{\uvec}$-geodesics $\bgeodu{x}$ and $\bgeodu{y}$ have a point in common they {\it coalesce}: namely,  if $\bgeodu{x}_m=\bgeodu{y}_n$ then  $\bgeodu{x}_{m+k}=\bgeodu{y}_{n+k}$ for all $k\ge 0$. 

Consider the following three events. 
\be\label{bgeod100} \begin{aligned} 
{\rm(i)} \quad  &\{\text{there exists a bi-infinite $\Bus{\uvec}$-geodesic}\} \\
{\rm(ii)} \quad &\{\text{$\exists x\in\Z^2$ such that   $\cC^\uvec(x)$ is infinite}\} \\
{\rm(iii)} \quad &\{\text{$\exists x,y\in\Z^2$ such that the  $\Bus{\uvec}$-geodesics $\bgeodu{x}$ and $\bgeodu{y}$  are disjoint}\} .
\end{aligned} \ee
The goal is to show that almost surely  none of these happen.  The first step is to show that they happen together, modulo the duality.

\begin{lemma}\label{bgeod-lm8}  Fix $\uvec\in\ri\Uset$.  Then   all three events in \eqref{bgeod100} have equal  probability. 

\end{lemma}

\begin{proof}   {\bf Step 1.} $\{{\rm(i)}\}=\{{\rm(ii)}\}$. 
%$\P\{{\rm(i)}\}>0\,\Longleftrightarrow\,\P\{{\rm(ii)}\}>0$. 
 For one direction, any point on a bi-infinite $\Bus{\uvec}$-geodesic has an infinite backward $\Bus{\uvec}$-cluster.  Conversely, suppose $\cC^\uvec(x_0)$ is infinite.  Then for each $m\in\Z_{>0}$ there exists $y(m)\in x_0+\Z_{\le0}^2$ such that $\bgeodu{y(m)}_{m}=x_0$.   From the finite paths $\bgeodu{y(m)}_{0,m}$,  a compactness argument  produces an infinite  backward path  $\{x_i\}_{i\le 0}$ such that $\bgeodu{x_i}_1=x_{i+1}$ for all $i<0$.   Extend  this infinite backward path to a bi-infinite $\Bus{\uvec}$-geodesic $x_\dbullet$ by defining $x_i=\bgeodu{x_0}_i$ for $i>0$. 
 
 Here is the compactness argument. 
%  By compactness, we can choose a subsequence $y_{m_j}$ such that for each $k>0$, as $j\to\infty$,  the $k$-step path segments $\bgeodu{y_{m_j}}_{m_j-k, m_j} $ converge to segments (denoted by) $x_{-k,0}$.   
  Choose nested subsequences of indices 
 $\{{m^1_j}\}_{j\ge 1} \supset \{{m^2_j}\}_{j\ge 1} \supset\dotsm\supset \{{m^k_j}\}_{j\ge 1}\supset\dotsm$ 
 % $\{y_{m^1_j}\}_{j\ge 1} \supset \{y_{m^2_j}\}_{j\ge 1} \supset\dotsm\supset \{y_{m^k_j}\}_{j\ge 1}\supset\dotsm$ 
 such that, for each $k$,   $m^k_1\ge k$ and the $k$-step path segments $\bgeodu{y(m^k_j)}_{m^k_j-k, m^k_j} $ converge to a path $x_{-k,0}$ as $j\to\infty$.   Convergent subsequences exist because the $k$-step segments  $\{\bgeodu{y(m)}_{m-k,m}\}_{m\ge k}$ lie in the finite set of $k$-step paths that end at $x_0$.  Since the subsequences are nested, the limits are consistent and form a single backward nearest-neighbor  path  $\{x_i\}_{i\le 0}$.  
 %  To complete the construction take the diagonal subsequence  $y(m_j)=y(m^j_j)$
 %This defines a backward nearest-neighbor  path  $\{x_i\}_{i\le 0}$.  
  Since the convergence happens on a discrete set,  for each $k$ there exists $j(k)<\infty$ such that  $\bgeodu{y(m^k_j)}_{m^k_j-k, m^k_j} =x_{-k,0}$ for $j\ge j(k)$. This implies that $\bgeodu{x_i}_1=x_{i+1}$ for all $i<0$.  

\medskip 

 {\bf Step 2.} $\P\{{\rm(iii)}\}\le \P\{{\rm(ii)}\}$.   Suppose event (iii) happens and let points $x_0, y_0\in\Z^2$ be such that  geodesics $\bgeodu{x_0}$ and $\bgeodu{y_0}$ are disjoint. By the limit in \eqref{bgeod98},  both coordinates $\bgeodu{x_0}_n\cdot \evec_1$ and  $\bgeodu{x_0}_n\cdot \evec_2$ increase to $\infty$ as $n\to\infty$, and the same for $\bgeodu{y_0}$.    By suitably redefining the initial points we can assume 
 that  $x_0$ and $y_0$ lie on the same antiodiagonal (that is, $x_0\cdot (\evec_1+\evec_2)=y_0\cdot (\evec_1+\evec_2)$) and  $x_0\cdot \evec_2<y_0\cdot \evec_2$, so that $\bgeodu{y_0}$  is   above and to the  left of $\bgeodu{x_0}$. 

 Let us say that a dual point $z\in\Z^{2*}$ lies between the two geodesics if the ray $\{z+t(\evec_2-\evec_1): t\ge 0\}$ hits a point of $\bgeodu{y_0}$ and the ray $\{z+t(\evec_1-\evec_2): t\ge 0\}$ hits a point of $\bgeodu{x_0}$.  
 
 % below $z$ lies  an $\evec_1$-directed arrow on  $\bgeodu{x_0}$ and above $z$ lies  an $\evec_1$-directed arrow on  $\bgeodu{y_0}$.  Precisely,  there exist  $k,\ell\ge0$ such that, with   $x=z-(\tfrac12,\tfrac12)-k\evec_2$ and    $y=z+(-\tfrac12, \tfrac12)+\ell \evec_2$,  we have  $\{x,x+\evec_1\}\in\bgeodu{x_0}$ and $\{y,y+\evec_1\}\in\bgeodu{y_0}$.
 
 For each $z\in\Z^{2*}$ that lies between the two geodesics, at least one of $z+\evec_1$  and $z+\evec_2$ also lies between the two geodesics.  For if $z+\evec_1$ does not lie between the two geodesics, then   edge $\{z,z+\evec_1\}$ must cross an edge of  $\bgeodu{x_0}$, and this edge is    $\{z+(\tfrac12,-\tfrac12), z+(\tfrac12,\tfrac12)\}$.  Similarly, if $z+\evec_2$ does not lie between the two geodesics, the   edge $\{z+(-\tfrac12,\tfrac12), z+(\tfrac12,\tfrac12)\}$   belongs to $\bgeodu{y_0}$.    Thus  if neither $z+\evec_1$  nor  $z+\evec_2$  lies between the two geodesics,  the two geodesics meet at the point $z+(\tfrac12,\tfrac12)$, contrary to the assumption of no coalescence. 
 
 Thus we can choose a semi-infinite path $\{z_m\}_{m\in\Z_{\ge0}}$ on the dual lattice such that $z_0\cdot (\evec_1+\evec_2)=y_0\cdot (\evec_1+\evec_2)$,  $z_{m+1}\in\{ z_m+\evec_1, z_m+\evec_2\}$ for all $m$, and the entire path $z_\bbullet$ lies between the   geodesics $\bgeodu{x_0}$ and $\bgeodu{y_0}$.   Since $\Bus{\uvec}$-geodesics and dual $\Bus{\uvec}$-geodesics never cross,  the (finite) dual geodesics $\dbgeodu{z_m}_{0,m}$ must also lie between the   geodesics $\bgeodu{x_0}$ and $\bgeodu{y_0}$. In particular, the endpoints $\{\dbgeodu{z_m}_{m}\}_{m\in\Z_{\ge0}}$  lie on the bounded antidiagonal segment between $x_0$ and $y_0$.    By compactness  there is a subsequence $z_{m_j}$ such that  the endpoint converges:  $\dbgeodu{z_{m_j}}_{m_j} \to z^*$.    Since this convergence happens on a discrete set, there exists some $j_0$ such that  $\dbgeodu{z_{m_j}}_{m_j}= z^*$ for all $j\ge j_0$. Thereby the (dual)  backward $\Bus{\uvec}$-cluster  $\cC^{*,\uvec}(z^*)$ is infinite. 
 
 We have shown that  event (iii) implies that event (ii) happens for  dual geodesics. By  the distributional equality of the families of $\Bus{\uvec}$-geodesics and dual $\Bus{\uvec}$-geodesics, the conclusion $\P\{{\rm(iii)}\}\le \P\{{\rm(ii)}\}$ follows.  
   
\medskip

  {\bf Step 3.} $\P\{{\rm(i)}\}\le \P\{{\rm(iii)}\}$.      Let $x$ and $y$ be two points on $\Z^2$ on  opposite sides of a bi-infinite dual $\Bus{\uvec}$-geodesic.  Geodesics $\bgeodu{x}$ and $\bgeodu{y}$ cannot cross the dual $\Bus{\uvec}$-geodesic (Lemma \ref{dgeod-lm1}), and hence cannot coalesce. 
   \end{proof} 
   
   \begin{theorem}\label{t:bg105} 
Fix $\uvec\in\ri\Uset$.  Then   all three events in \eqref{bgeod100} have zero probability. 
\end{theorem} 

\begin{proof}   This theorem follows from  Lemma \ref{bgeod-lm8} and 
\be\label{bg89} \P\{\text{\rm there exists a bi-infinite $\Bus{\uvec}$-geodesic}\}=0.  \ee 
%%%%%%%%%%%%   MISTAKEN ORIGINAL   %%%%%%
% in two different ways.  
% 
% \medskip 
% 
% (a) A bi-infinite $\Bus{\uvec}$-geodesic  would  split  the dual graph $\cT_\uvec^*$ into two disjoint pieces  because by Lemma \ref{dgeod-lm1}  $\Bus{\uvec}$-geodesics and dual $\Bus{\uvec}$-geodesics do not cross each other.    This happens with probability zero because by the already proved Theorem \ref{t:main2} $\cT_\uvec^*$ is a tree and hence connected.   
% 
% \medskip 
% 
% (b) Alternatively, 
 %%%%%%%%%%%%%%
To  prove \eqref{bg89} 
 we use the solution of the midpoint problem to prove that a bi-infinite $\Bus{\uvec}$-geodesic goes through the origin with probability zero.  Suppose $\{x_n\}_{n\in\Z}$ is a bi-infinite $\Bus{\uvec}$-geodesic with $x_0=0$.     To apply Theorem \ref{th:midpt} to $u_n=x_{-n}$, $z_n=0$ and $v_n=x_n$ we need the limits 
  \be\label{bgeod104}    \frac{x_{-n}}n\to -\uvec\quad\text{and}\quad  \frac{x_{n}}n\to \uvec \ee
  almost surely on the event where a bi-infinite $\Bus{\uvec}$-geodesic through the origin exists. 
  
  The second limit of  \eqref{bgeod104}  is in \eqref{bgeod98}.  The backward limit ${x_{-n}}/n\to -\uvec$ is proved by the same argument.  Namely, since $x_{-n,0}$ is a (finite) $\Bus{\uvec}$-geodesic (that is, $\bgeodu{x_{-n}}_j=x_{-n+j}$ for $0\le j\le n$),  
 \eqref{bgeod93} applies and gives 
 \[  G_{x_{-n},0} =   \Bus{\uvec}_{x_{-n},0} + \Yw_0 . \]
 The uniform passage time limit \eqref{sh89} applies to the southwest LPP process to give 
 \[    G_{x_{-n},0}  =  \gpp(-x_{-n}) + o(n)\qquad\text{almost surely.}  \]
 The uniform ergodic theorem for cocycles (Theorem   \ref{cc-ergthm}) gives 
\[   \Bus{\uvec}_{x_{-n},0}  = \nabla\gpp(\uvec)\cdot (-x_{-n}) + o(n) \qquad\text{almost surely.}  \] 
These almost sure asymptotics  and strict concavity of $\gpp$  in the form \eqref{bg-gpp7} then  imply that the first limit in \eqref{bgeod104} holds almost surely on the  event where a bi-infinite $\Bus{\uvec}$-geodesic $x_\bbullet$ through $x_0=0$ exists.

Since $x_{-n,n}$ is a geodesic through the origin,   we have $0\in\pi^{x_{-n}, x_n}$ for all $n>0$ on the event where the bi-infinite geodesic $x_\bbullet$ goes through the origin.  By Theorem \ref{th:midpt} this  event must have probability zero. 
   \end{proof} 
   
\subsection{Completion of the proofs}    

\begin{proof}[Proof of Theorem \ref{t:main1}]  
 Lemma \ref{Bgeod-lm}(iii) implies that almost surely there is a  unique $\uvec$-directed semi-infinite geodesic out of $x$, namely the $\Bus{\uvec}$-geodesic  $\bgeodu{x}$. Its construction \eqref{bg6} shows that it is a Borel function of the random variables $\Bus{\uvec}_{x,y}$ which in turn are Borel functions of $\Yw$.     Theorem \ref{t:bg105} gives the almost sure  coalescence and non-existence of a bi-infinite geodesic.  (A bi-infinite $\uvec$-directed geodesic $(x_i)_{i\in\Z}$  must also be a   $\Bus{\uvec}$-geodesic because by  Lemma \ref{Bgeod-lm}(iii), $(x_i)_{i\ge\ell}=\bgeodu{x_\ell}$  for each $\ell\in\Z$.) 
 \end{proof}

 \begin{proof}[Proof of Theorem \ref{t:main2}]   Process $\wt Y^\uvec$ was defined in \eqref{Ytil4} and hence by  \eqref{915} satisfies 
\[  \wt Y^\uvec_x=\Xw^{\uvec}_{-x}=\Bus{\uvec}_{-x-\evec_1, -x}\wedge \Bus{\uvec}_{-x-\evec_2, -x}.    \]
The i.i.d.\ Exp(1) distribution of $\Xw^{\uvec}$ gives the same to $\wt Y^\uvec$,  limit \eqref{v:855} shows that $\wt Y^\uvec$ is a Borel function of $\Yw$, and the second equality of the display above implies that $\wt Y^\uvec_x(\theta_y\w)=\wt Y^\uvec_{x-y}(\w)$. 

  Lemma \ref{dgeod-lm1}  and definition \eqref{T*} imply that $\cT_\uvec^*=\bigcup_{z\in\Z^{2*}} \dbgeodu{z}$.  Tracing through  definition \eqref{Ttil} of  $\wt\cT_\uvec$, definition \eqref{bg22} of dual geodesics, and \eqref{bg26} gives the equivalence 
  \[     \{x,x+\evec_1\}\in\wt\cT_\uvec \ 
  \Longleftrightarrow \   \wt B^\uvec_{x,x+\evec_1}\le \wt B^\uvec_{x,x+\evec_2}. \]
  A similar argument gives  $\{x,x+\evec_2\}\in\wt\cT_\uvec \ 
  \Longleftrightarrow \   \wt B^\uvec_{x,x+\evec_1}> \wt B^\uvec_{x,x+\evec_2}.$   The proof of Lemma \ref{dgeod-lm2} observed that $\wt B^\uvec$ is the Busemann function of the   LPP process   with weights $\wt Y^\uvec$.  Hence   \eqref{T3} applied to  weights  $\wt Y^\uvec$ implies that  $\wt\cT_\uvec$ is the tree of semi-infinite $\uvec$-directed geodesics for this LPP process.  
\end{proof}

 \begin{proof}[Proof of Theorem \ref{t:main3}]  
% Let 
% \[  \arrowu{x}=\bgeodu{x}_1-x  = \begin{cases}  \evec_1, &\Bus{\uvec}_{x,x+\evec_1}\le  \Bus{\uvec}_{x,x+\evec_2} \\   \evec_2, &\Bus{\uvec}_{x,x+\evec_2}<  \Bus{\uvec}_{x,x+\evec_1} 
% \end{cases} \]
%  be the first step of the  $\Bus{\uvec}$ Busemann geodesic \eqref{bg6} started at $x$.  
%  Define $a_j\in\{1,2\}$ by  $\arrowu{(N+j,-j-1)}=\evec_{a_j}$.  
  Define $a_j\in\{1,2\}$ by
 \[  a_j = \begin{cases}   1, &\Bus{\uvec}_{(N+j,-j-1) ,(N+j+1,-j-1)}\le  \Bus{\uvec}_{(N+j,-j-1) ,(N+j,-j)} \\   2, &\Bus{\uvec}_{(N+j,-j-1) ,(N+j,-j)}<  \Bus{\uvec}_{(N+j,-j-1) ,(N+j+1,-j-1)} .  
 \end{cases} \] 
 Property (a) in Definition \ref{v:d-exp-a} applied to the down-right path 
 \[  y_{2j}=(N+j,-j-1), \quad y_{2j+1}=(N+j+1,-j-1)  \] 
 implies that     $\{a_j\}_{j\in\Z}$ are i.i.d.\ random variables with marginal distribution 
   \[   \P(a_j=1)=\alpha=1-\P(a_j=2). \]
 The process $\{\xi_j\}$    is obtained from the connection 
 \[  \xi_j=\begin{cases}  s\\c\\h\\v \end{cases}  
 \quad\text{if}\quad   
 (a_{j-1},a_j)=\begin{cases}  (2,1)\\(1,2)\\(1,1)\\(2,2).  \end{cases}
 \]
 Thus  $\{\xi_j\}$  has the distribution of the Markov chain $X_j=(a_{j-1},a_j)$, after relabeling   the states as above.   
   \end{proof}

%\note{Next lemma does not seem to be true.  Is anything true?
%\begin{lemma} 
%Fix $v\in\Z^2$.  Suppose there are  two distinct points   $u_1$ and $u_2$ in the south-west quadrant $(v-\Z_+^2)\setminus\{v\}$ such that the geodesics $\bgeodu{u_1}_\bbullet$ and $\bgeodu{u_2}_\bbullet$ meet for the first time at $v$. {\rm(}Equivalently,  $v$ and $v-\evec_1$ lie on one the geodesics, and $v$ and $v-\evec_2$ on the other.{\rm)}  Then the dual geodesic $\dbgeodu{v-\duale}_\bbullet$ is the competition interface for the LPP process $\{\Gpp_{u,v}: u\le v\}$. 
%\end{lemma}
%}

\section{Increment-stationary LPP and competiton interface} \label{s:cif} 
%\begin{remark}[Competition interface]   
 This section explains how  $\Bus{\uvec}$ represents a LPP process with boundary conditions  and how the paths  $\bgeodu{x}$ and $\swbgeodu{x}$ function both as  geodesics and competition interfaces, depending on whether  the LPP uses weights $\Yw$ or  $\Xw^\uvec$.     Fix $\uvec\in\ri\Uset$.  Fix also a down-right path  $\cY=(y_k)_{k\in\Z}$  on   $\Z^2$, that is,  a sequence in $\Z^2$  such that    $y_k-y_{k-1}\in\{\evec_1,-\evec_2\}$ for all $k\in\Z$.  Let $\cH^\pm$ be as in \eqref{H-}--\eqref{H+} and define $\wtcH^\pm=\cY\cup\cH^\pm$.   $\cY$ serves as a boundary and the LPP processes will be defined in the regions $\wtcH^\pm$. 
 
 Let $\abs\pi$ denote the Euclidean length (number of edges) of a nearest-neighbor lattice path.  
 For $x\in\wtcH^+$, let $\Pi^{\cY, x}$ be the set of up-right paths   $\pi=\pi_{0,n}=(\pi_i)_{i=0}^n$ 
 of any length $n=\abs\pi$ that go from $\cY$ to $x$ and that lie in $\cH^+$ except for the initial point on $\cY$:  
 \[  \Pi^{\cY, x}=\{\pi  :  \pi\in\Pi_{\pi_0, x},    \pi_0\in\cY, \pi_{1,\abs\pi}\subset\cH^+\}. 
% \Pi^{\cY, x}=\{\pi=\pi_{0,n} : n\in\Z_{>0},  \pi_0\in\cY, \pi_{1,n}\subset\cH^+, \pi_{0,n}\in\Pi_{\pi_0, x}\}. 
 \]
 For $x\in\wtcH^+$ define the LPP process 
 \be\label{lpp+}   H^+_x=\sup_{\pi\,\in\,\Pi^{\cY,x}}\Bigl\{   \Bus{\uvec}_{y_0,\pi_0}  +\sum_{i=1}^{\abs\pi} \Xw^\uvec_{\pi_i}   \Bigr\} . 
 \ee
 In the degenerate case  $x\in\cY$ and  $H^+_x=\Bus{\uvec}_{y_0,x}$.  
 The set of paths maximized over can be finite (for example in case $\lim_{k\to-\infty} y_k\cdot \evec_2=\infty$ and $\lim_{k\to\infty} y_k\cdot \evec_1=\infty$) or infinite (for example if $y_k=k\evec_1$ is the $x$-axis).   
 %The choice of the base point $0$ in $\Bus{\uvec}_{0,\pi_0}$ in \eqref{lpp+} is arbitrary.  Replacing it with another point (for example $y_0$) changes the process $H^+$ by an additive constant.   
 The random variables $\Xw^\uvec_x$ over $x\in\cH^+$ and $\Bus{\uvec}_{y_k,y_{k+1}}$ on $\cY$  are all independent, so $H^+$ is an LPP process that uses independent weights.  To ensure unique geodesics, we restrict ourselves to the full-measure event on which 
 \be\label{Yass} \text{no two nonempty sums of distinct $\{\Xw^\uvec_x\}_{x\in\cH^+}$ and $\{\Bus{\uvec}_{y_k,y_{k+1}}\}_{k\in\Z}$ agree.} 
 \ee
 
% We show below that for each $x\in\cH^+$ there is a unique maximizing path in \eqref{lpp+}.   
 
 A combination of \eqref{915} and \eqref{bg18},  as in the proof of Lemma \ref{Bgeod-lm}(i), shows that the LPP process $H^+$ coincides with $\Bus{\uvec}$ and that  the southwest geodesics  are the geodesics in this process.     
 
 \begin{proposition} \label{pr:Hgeod}  Fix $x\in\cH^+$.  
  Let $n=\min\{i\ge 0: \swbgeodu{x}_i\in\cY\}$.   Then  $\{\pi^{+,x}_i=\swbgeodu{x}_{n-i}\}_{0\le i\le n}$ is the unique maximizing path in \eqref{lpp+} and 
%  $H^+_x =   \Bus{\uvec}_{0,\pi^{+,x}_0}  +\sum_{i=1}^{n} \Xw^\uvec_{\pi^{+,x}_i}   =   \Bus{\uvec}_{0,x}.$
  \be\label{lpp+3}   H^+_x
=   \Bus{\uvec}_{y_0,\pi^{+,x}_0}  +\sum_{i=1}^{n} \Xw^\uvec_{\pi^{+,x}_i}   =   \Bus{\uvec}_{y_0,x}.
 %  \Bus{\uvec}_{0,\swbgeodu{x}_n}  +\sum_{i=1}^{n-1} \Xw^\uvec_{\swbgeodu{x}_i}   =   \Bus{\uvec}_{0,x}. 
    \ee
 \end{proposition} 
 
% \begin{proof}  Combination of \eqref{915} and \eqref{bg18}, argued  as in the proof of Lemma \ref{Bgeod-lm}(i).    \end{proof}
 
 For $A\subset\cY$, let  $\cH^+_A=\{x\in\wtcH^+: \pi^{+,x}_0\in A\}$ denote the set of points $x$ whose geodesic emanates from $A$.  Fix two adjacent points $y_m, y_{m+1}$ on $\cY$.  Decompose    $\cH^+= \cH^+_{y_{-\infty, m}}  \cup \cH^+_{y_{m+1,\infty}}$ according to whether the geodesic emanates from $\{y_k\}_{k\le m}$ or $\{y_k\}_{k\ge m+1}$.    The two regions $\cH^+_{y_{-\infty, m}}$ and $\cH^+_{y_{m+1,\infty}}$ are separated by an up-right path $\varphi^+=(\varphi^+_n)_{n\ge 0}$ called the {\it competition interface}: 
\be\label{cif-ne} \begin{aligned}  
\varphi^+_0&=\begin{cases}  y_m,  &\Bus{\uvec}_{y_0,y_m}< \Bus{\uvec}_{y_0,y_{m+1}}\\
 y_{m+1},  &\Bus{\uvec}_{y_0,y_{m+1}}< \Bus{\uvec}_{y_0,y_{m}} 
\end{cases}  \\
\text{and for $k\ge 0$}\qquad 
\varphi^+_{k+1}&=\begin{cases} \varphi^+_k+\evec_1,  &H^+_{\varphi^+_k+\evec_1}< H^+_{\varphi^+_k+\evec_2}\\
\varphi^+_k+\evec_2,  &H^+_{\varphi^+_k+\evec_2}< H^+_{\varphi^+_k+\evec_1}. 
\end{cases}
\end{aligned} \ee
One can check inductively that for each $n\in\Z_{\ge0}$, $\varphi^+_n$ is the unique point on its  antidiagonal $\{x\in\wtcH^+: x\cdot (\evec_1+\evec_2)=\varphi^+_0\cdot (\evec_1+\evec_2)+n \}$ that satisfies   $\varphi^+_n+\Z_{>0}\evec_2\subset \cH^+_{y_{-\infty,m}}$ and $\varphi^+_n+\Z_{>0}\evec_1\subset \cH^+_{y_{m+1,\infty}}$.    
 Comparison of  \eqref{bg6}  and  \eqref{cif-ne}, with an appeal to \eqref{lpp+3},  proves the next  characterization of $\varphi^+$. 
 \begin{proposition}\label{pr:Hcif} 
 $\varphi^+=\bgeodu{\varphi^+_0}$. 
 \end{proposition} 
 
 An analogous  LPP process is defined for  $x\in\wtcH^-$  with weights $\Yw$:
  \be\label{lpp-}   H^-_x=\sup_{\pi\,\in\,\Pi^{x,\cY}}\Bigl\{ \, \sum_{i=0}^{\abs\pi-1} \Yw_{\pi_i}  +  \Bus{\uvec}_{\pi_{\abs\pi}, y_0}    \Bigr\} 
 \ee
 where $\Pi^{x,\cY}$ is  the set of up-right paths   from $x$ to $\cY$   that lie in $\cH^-$ except for their final  point on $\cY$.  This time $H^-_x=\Bus{\uvec}_{x, y_0}$,  the part of $\bgeodu{x}$ between $x$ and $\cY$ is the geodesic,  and competition interfaces are southwest geodesics $\swbgeodu{y}$ emanating from points $y\in\cY$.  We omit the details. 

From the results of this section  we can derive Lemma 4.4 of \cite{bala-cato-sepp} as a special case.  Namely,  fix $(m,n)\in\Z_{>0}^2$ and take  $\cH^-$ to be the southwest quadrant bounded on the north and east by the path 
$ y^-_k=(m,n)-k^+\evec_2-k^-\evec_1$.  In terms of the  LPP process $H^-$ and the weights $Y$   in \eqref{lpp-} above, the ``reversed process'' $G^*$ and weights $\w^*$  in  \cite{bala-cato-sepp}  correspond to $G^*_{ij}=H^-_{(m-i,n-j)}$ and $\w^*_{ij} = \Yw_{(m-i,n-j)}$.     The competition interface that emanates from the ``origin'' $(m,n)$ for $H^-$ in \eqref{lpp-}  is the southwest geodesic  $\swbgeodu{(m,n)}$.   According to Proposition \ref{pr:Hgeod} above 
%Lemma \ref{Bgeod-lmX} 
this path is also the geodesic for LPP process  $H^+$ constructed with coordinate axes boundary $ y_k=k^+\evec_1+k^-\evec_2$.  This is exactly what Lemma 4.4 of \cite{bala-cato-sepp} says. 
% \section{Appendix} 

 \appendix
 
 \section{Planar monotonicity} 
 
Planar  LPP increments possess monotonicity properties.  The lemma below can be found proved as Lemma 4.6 in  \cite{sepp-cgm-18}.  Let the LPP process $\Gpp$ be defined by \eqref{v:GY7} and define increments 
\[	 I_{x,v} =  \Gpp_{x,v} -  \Gpp_{x+e_1,v} \qquad  \text{ and } \qquad   J_{y,v} =  \Gpp_{y,v} -  \Gpp_{y+e_2,v}\,. \]

\begin{lemma}\label{lm-til}  For   $x,y,v\in\Z^2$ such that  $x\le v-e_1$ and $y\le v-e_2$ 
		\be
		\label{til-8}
			 I_{x,v+e_2} \ge  I_{x,v} \ge  I_{x,v+e_1} 
			\qquad  \text{ and } \qquad 
			 J_{y, v+e_2} \le  J_{y,v} \le  J_{y, v+e_1}\,.
		\ee
\end{lemma}

\section{Cocycle ergodic theorem}  
 
Recall Definition  \ref{def:cK}.  Covariant integrable cocycles satisfy a uniform ergodic theorem, sometimes also called a shape theorem.

\begin{theorem}  \label{cc-ergthm} Let  $F\in\cK$ be such that   $\E[F(x,y)]=0$  $\forall x,y\in\Z^2$.    Assume that there  exists   a function $\overline F:\Omega\times\{\evec_1,\evec_2\}\to\R$ such that, $\P$-almost surely and for $k\in\{1,2\}$, 
   $F(\w,0,\evec_k)\le \overline F(\w,\evec_k) $  and 
  \be \label{cc-ass1}  \begin{aligned}  
  \varlimsup_{\delta\searrow0 } \;\varlimsup_{n\to\infty} \; \max_{\abs{x}_1\le n}\;  \frac1n \sum_{0\le i\le n\delta}  \abs{\overline F(\theta_{x+i\evec_k}\w, \evec_k)} =0 . 
\end{aligned}\ee 
   Then  
\[ \lim_{n\to\infty}  \max_{\abs{x}_1\le n} \frac{\abs{F(\w,0,x)}}n \;=\; 0 \qquad\P\text{-a.s.} \]

\end{theorem} 
For a proof see Appendix A.3 of \cite{geor-rass-sepp-yilm-15}.  
A sufficient condition for   limit \eqref{cc-ass1}  is that $\E\abs{\overline F(\w, \evec_k)}^{2+\e} <\infty$ for some $\e>0$ and the shifts of $\overline F$ have  finite range of dependence: namely,    $\exists r_0<\infty $ such that  if  $\abs{x_i-x_j}\ge r_0$ for each pair $i\ne j$, then  
$\{(\overline F(\theta_{x_i}\w, \evec_k))_{k\in\{1,2\}}: 1\le i\le m\}$ is a sequence of $m$ independent  random vectors.  
%\note{$z$ or $\tilde z$ above? Any role for group $\cG$?}  

%
%\bibliographystyle{plain}
%
%\bibliography{growthrefs}
%

\end{document}